\documentclass[11pt]{amsart}
\usepackage{amsmath}
\usepackage{amssymb}
\usepackage{amsfonts}
\usepackage{latexsym}
\usepackage{fancyhdr}
\usepackage{graphics}
\usepackage{color}
\newcommand{\mb}{\mathbb}
\newcommand{\mc}{\mathcal}

\newcommand{\M}{\mc M}

\newcommand{\D}{\mc D}

\def\H{\mc H}

\newcommand{\hh}{\mc H}

\newcommand{\N}{{\mc N}}

\newcommand{\SC}{\mathcal{S}({\mb R})}
\newcommand{\SSS}{\mathcal{S}^\times({\mb R})}

\newcommand{\LDD}{{\mc L}(\D,\D^\times)}
\newtheorem{defn}{Definition}[section]
\newtheorem{prop}[defn]{Proposition}
\newtheorem{thm}[defn]{Theorem}
\newtheorem{lemma}[defn]{Lemma}
\newtheorem{cor}[defn]{Corollary}
\newtheorem{example}[defn]{Example}
\newtheorem{rem}[defn]{Remark}
\def\x{\relax\ifmmode {\mbox{*}}\else*\fi}
\newcommand{\beex}{\begin{example}$\!\!${\bf }$\;$\rm }
\newcommand{\enex}{ \end{example}}
\newcommand{\berem}{\begin{rem}$\!\!${\bf }$\;$\rm }
\newcommand{\enrem}{ \end{rem}}

\newcommand{\bedefi}{\begin{defn}$\!\!${\bf }$\;$\rm }
\newcommand{\findefi}{\end{defn}}

\newcommand{\ip}[2]{\left\langle {#1}|{#2}\right\rangle}
\newcommand{\RN}{{\mb R}}

\newcommand{\LD}{{\mathcal L}^\dagger(\D)}
\begin{document}

\title[Distribution frames]{Distribution frames and bases}%

\author{Camillo Trapani}%
\address{Dipartimento di Matematica e Informatica, Universit\`a
di Palermo, I-90123 Palermo (Italy)} \email{camillo.trapani@unipa.it}
\author{Salvatore Triolo}%
\address{Dipartimento DEIM, Universit\`a di Palermo, I-90123 Palermo (Italy)}%
\email{salvatore.triolo@unipa.it}
\author{Francesco Tschinke}%
\address{Dipartimento di Matematica e Informatica, Universit\`a
di Palermo, I-90123 Palermo (Italy)} \email{francesco.tschinke@unipa.it}

\begin{abstract}
\noindent In this paper we will consider, in the abstract setting of rigged Hilbert spaces, distribution valued functions and we will investigate conditions for them to constitute a ''continuous basis" for the smallest space $\D$ of a rigged Hilbert space. This analysis requires suitable extensions of familiar notions as those of frames, Riesz bases and orthonormal bases. A motivation for this study comes from the Gel'fand-Maurin theorem which states, under certain conditions, the existence of a family of generalized eigenvectors of an essentially self-adjoint operator on a domain $\D$ which acts like an orthonormal basis of the Hilbert space $\H$. The corresponding object will be called here a {\em Gel'fand distribution basis}. The main results are obtained in terms of properties of a conveniently defined {\em synthesis operator}.

\smallskip
\noindent{\bf Keywords:} Distributions; Rigged Hilbert spaces; Frames; Bases.
\end{abstract}

\subjclass[2010]{Primary 47A70; Secondary 42C15, 42C30 }
\maketitle
\section{Introduction}
As is well-known, the Gel'fand-Maurin theorem \cite{gelf} states that if $\D$ is a domain in a Hilbert space $\H$ which is a nuclear space under a certain topology $\tau$ and $A$ is an essentially self-adjoint operator in $\D$ which maps $\D[\tau]$ into $\D[\tau]$ continuously, then $A$ admits a {\em complete  set of generalized eigenvectors} (for a complete proof see also \cite{gould}).
In this case $A$ has a continuous extension $\hat{A}$ given by the conjugate duality (the adjoint, in other words; i.e. $\hat{A}=A^\dag$) from the conjugate dual space $\D^\times$ into itself.
A \textit{generalized eigenvector} of $A$, with eigenvalue $\lambda\in\mathbb C$, is then an eigenvector of $\hat{A}$; that is, a conjugate linear functional $\Phi_\lambda\in\D^\times$ such that:
$$
\ip{\Phi_\lambda}{Af}=\lambda \ip{\Phi_\lambda}{f},\quad\forall f\in\D.
$$
The above equality can be read as $\hat{A}\Phi_\lambda=A^\dag\Phi_\lambda=\lambda\Phi_\lambda$.
The completeness of the set $\{\Phi_\lambda; \lambda \in \sigma(\overline{A})$\} is understood in the sense that  the Parseval identity holds, that is
$$ \|f\|= \left(\int_{\sigma(\overline{A})} |\ip{f}{\Phi_\lambda}|^2 d\lambda\right)^{1/2}, \quad \forall f\in \D,$$
which also gives account of a kind of {\em orthogonality} of the $\Phi_\lambda$'s (a family $\{\Phi_\lambda\}$ with this property will here be called a {\em Parseval frame}).
To each $\lambda$ there corresponds the subspace $\D^\times_\lambda \subset \D^\times$ of all  generalized eigenvectors whose eigenvalue is $\lambda$. For all $f\in\D$ it is possible to define a linear functional $\tilde{f}_\lambda$ on $\D^\times_\lambda$ by $\tilde{f}_\lambda(\Phi_\lambda):=\ip{\Phi_\lambda}{f}$ for all $\Phi_\lambda\in\D_\lambda^\times$. The correspondence $\D\rightarrow\D_\lambda^{\times\times}$ defined by $f\mapsto \tilde{f}_\lambda$ is called the \textit{spectral decomposition of the element $f$ corresponding to $A$}. If $\tilde{f}_\lambda\equiv 0$ implies $f=0$ (i.e. the map $f\mapsto \tilde{f}_\lambda$ is injective) then $A$ is said to have \textit{a complete system of generalized eigenvectors}.

A consequence of the Gel'fand-Maurin theorem is the \textit{spectral expansion theorem} that plays an important role in quantum mechanics (see for instance \cite{messiah}): given the rigged Hilbert space
$$
\mathcal S(\mathbb R^n)\hookrightarrow  L^2(\mathbb R^n)\hookrightarrow\mathcal S^\times(\mathbb R^n),
$$
 ($\hookrightarrow $ stands for a continuous
embedding) the  Hamiltonian operator $H$ is  an essentially self-adjoint operator on $\mathcal S(\mathbb R^n)$,  with self-adjoint extension $\overline{H}$ on the domain $\D(\overline H)$, dense in $L^2(\mathbb R^n)$.
The spectral expansion theorem  in the form usually written by physicists asserts
 that there exists  an orthonormal system of
 eigenvectors $\{u_{m,s}\}\subset L^2(\mathbb R^n)$,
a countable family of  pairs
$\{(\sigma_{ i},\mu_i)\}$,  where $\sigma_{ i}\subset\mathbb R$ and $\mu_i$ are continuous measures on $\sigma_{ i}$,
and \textit{some families of generalized eigenvectors in $\mathcal S^\times(\mathbb R^n)$}:
 $\{u_{1,\alpha}\}_{\alpha\in\sigma_{1}}$, $\{u_{2,\alpha}\}_{\alpha\in\sigma_{2}}$, ..., $\{u_{i,\alpha}\}_{\alpha\in\sigma_{i}}$,...
 such that, for every $f\in  L^2(\mathbb R^n)$ the following decomposition holds:
$$
f=\sum_{m,s} c_{m,s}u_{m,s}+\sum_i\int_{\sigma_{c,i}} c_i(\alpha)u_{i,\alpha}d\mu_i(\alpha).
$$
Since this expansion is unique, the generalized eigenvectors $u_{i,\alpha}$ can be considered as a {\em distribution basis}.

The previous discussion makes clear that the natural environment to consider this kind of questions is that of {\em rigged Hilbert spaces} or {\em Gel'fand triplets} constructed by endowing a dense subspace $\D$ of Hilbert space $\H$ with a topology finer than that induced by $\H$ and taking the conjugate dual $\D^\times$; this produces the triplet of spaces
\begin{equation}
\D[t] \hookrightarrow  \H \hookrightarrow\D^\times[t^\times].
\end{equation}
Some results on bases in rigged Hilbert spaces were already obtained (for the discrete case) in \cite{gb_ct_riesz}.

 A typical, albeit elementary, example of this situation is provided by the derivative operator in rigged Hilbert spaces of distributions. Let $$ \SC \subset L^2({\mb R})\subset \SSS$$
be the rigged Hilbert space constructed starting from the Schwartz space $\SC$ of rapidly decreasing $C^\infty$-functions on ${\mb R}$. The space $\SC$ is nuclear and the operator $A=i\frac{d}{dy}:\SC\rightarrow\SC$, is essentially self-adjoint in $\SC$. The operator $A$ has the family of functions $\omega_x(y)= \frac{1}{\sqrt{2\pi}}e^{-ixy}$, $x\in {\mb R}$ as a complete system of generalized eigenvectors ($\omega_x$ is a generalized eigenvector with eigenvalue $x$), and the spectral decomposition of  $f$ is defined by the Fourier transform of $f\in \SC$. It is clear that each function $\omega_x$ can be viewed as a regular distribution of $\SSS$, in the sense that, for every fixed $x\in {\mb R}$,
$$ \check{\phi}(x) = \ip{\phi}{\omega_x} =\int_\RN \phi(y)\overline{\omega_x(y)}dy =\frac{1}{\sqrt{2\pi}}\int_\RN \phi(y)e^{ixy}dy$$
defines a continuous linear functional on $\SC$.

If  $\Phi\in\SSS$ is the regular distribution determined by a function $g\in L^2(\mathbb R)$, i.e., $\Phi_g(\phi):=\int_\RN g(y)\phi(y) dy$, then,
\begin{equation}\label{eqn_fourier}
\Phi_g(\phi)=\int_\RN\phi(y) {\Bigl(}\frac{1}{\sqrt{2\pi}}\int_\RN\hat{g}(x)e^{ixy}dx{\Bigr )}dy=\int_\RN\hat{g}(x) \check{\phi}(x)dx=\int_\RN\hat{g}(x)\ip{\phi}{\omega_x}dx.
\end{equation}
This means that $\Phi_g$ can be {\em expanded} in terms of the distributions $\omega_x$, $x \in {\mb R}$ that can be considered as basis vectors.
The Fourier-Plancherel theorem ensures of the completeness of the family $\{\omega_x\}$ in the sense described above.

The family $\{\omega_x;\,x \in {\mb R}\}$ of the previous example can be regarded, obviously, as the range of a weakly measurable function $\omega: {\mb R}\to \SSS$ which allows a representation \eqref{eqn_fourier}  of any $\phi\in \SC$ enjoying the completeness in the sense of Gel'fand. We will take it as a prototype of what we will call {\em Gel'fand distribution bases}.

Orthonormal bases in a Hilbert space are nowadays considered just as particular cases of more general families (Riesz bases, frames, semi-frames, etc) having in common the possibility of {\em reconstructing} vectors of the space as the superposition of more 'elementary' vectors, renouncing often to the uniqueness of the representation. Bases and frames have been studied in harmonic analysis in various spaces of functions and distributions and the appearance of some generalization in rigged Hilbert spaces (or even in more general structures) has also been recorded \cite{jpa_ct_rppip, gb_ct_riesz,feich2, Kpet}. The aim of the present paper is to propose possible extension of some of these notions in rigged Hilbert spaces and study their mutual relationships. The advantage of this approach is probably in its flexibility: as we have seen in the example of the Fourier transform, expansions of vectors in term of objects that live beyond the Hilbert space can be useful and some bonds which usually come from the geometry of the Hilbert space may disappear. For instance,  in a separable Hilbert space orthonormal bases as well as Riesz bases are both necessarily countable and also in more general situations Riesz bases cannot be {\em continuous}, but they are {\em discrete}, in some sense \cite{hosseini, jacobsen, bal_ms_2} (while continuous frames in Hilbert spaces do really exist). The corresponding objects we are going to introduce in our set-up (i.e., Gel'fand distribution bases and Riesz distribution bases, frames) can be continuous as it happens for the family of functions $\omega_x(y)= \frac{1}{\sqrt{2\pi}}e^{-ixy}$, $x\in {\mb R}$, which is indeed a Gel'fand distribution basis of $L^2(\RN)$ (Example \ref{ex_3.15} below). The drawback of going beyond the Hilbert space consists mainly in the loss of regularity of the families of functions forming the frame (or basis) we are dealing with. This is the price one has to pay and deciding if it is worth or not adopting this approach depends essentially on the applications one is considering.

The paper is organized as follows. After some preliminaries (Section \ref{sect_2}), we study $\D^\times$-valued measurable functions, called {\em distribution functions}, which are Bessel maps with respect to the topology of $\D$ and distinguish the case when they are bounded with respect to the Hilbert norm. Among the latter ones, we consider {\em distribution frames}: as in the standard case they produce a {\em reconstruction formula} which allows to express every element of $\D$ in terms of elements of the frame and of its dual frame. This is done in Section \ref{sect_3}, where we also look for more suitable substitutes of orthonormal bases in this framework.
These are the Gel'fand distribution bases discussed above. A further step leads to consider Riesz distribution bases. What essentially distinguishes the various cases is the behavior of the synthesis operator related to each of them. In Section \ref{sect_4} we analyze operators defined by Gel'fand and Riesz distribution bases and examine, more generally the case where the action of an operator can be described by a Bessel distribution map in the spirit of \cite{feich, gavruta}.

Section \ref{sect_5} is devoted to the study of {\em coefficient spaces} defined by general pairs of distribution functions and we show that they enjoy nice duality properties exactly when a generalized {\em frame operator} is a topological isomorphism of $\D$ onto $\D^\times$. This leads to introducing the notion of {\em compatible pairs} of distribution valued function which could be taken as a reasonable generalization of the notion of {\em reproducing pairs} introduced in \cite{bal_ms_0, bal_ms} and studied also in \cite{jpa_ms_ct, jpa_ct_rppip}. It is worth mentioning that Feichtinger and Zimmermann \cite[Section 3.5]{feich 3} introduced objects similar to distribution
frames and compatible pairs in their study of dual Gabor systems
with windows taken from modulation spaces. These in fact consist of tempered distributions whose short-time Fourier transforms enjoy certain regularity conditions (see \cite[Section 8.3]{at_pip} for a short description).

\medskip

\section{Preliminary definitions and facts} \label{sect_2}
  Let $\D$ be a dense subspace of  $\H$.  A locally convex topology $t$ on $\D$ finer than the topology induced by the Hilbert norm defines, in standard fashion,
a {\em rigged Hilbert space}
\begin{equation}\label{eq_one_intr}
\D[t] \hookrightarrow  \H \hookrightarrow\D^\times[t^\times],
\end{equation}
 where $\D^\times$  is the vector space of all continuous conjugate
linear functionals on $\D[t]$, i.e., the conjugate dual of $\D[t]$,
endowed with the {\em strong dual topology} $t^\times=
\beta(\D^\times,\D)$, which can be defined by the seminorms
\begin{equation}\label{semin_Dtimes}
q_\M(F)=\sup_{g\in \M}|\ip{F}{g}|, \quad F\in \D^\times,
\end{equation}
where $\M$ is a bounded subsets of $\D[t]$.

Since the Hilbert space $\H$ can be identified  with a
subspace of $\D^\times[t^\times]$, we will systematically read
\eqref{eq_one_intr} as a chain of topological inclusions: $\D[t]
\subset  \H \subset\D^\times[t^\times]$.  These identifications
imply that the sesquilinear form $B( \cdot , \cdot )$ which puts $\D$
and $\D^\times$ in duality is an extension of the inner product of
$\H$;
 i.e. $B(\xi, \eta) = \ip{\xi}{\eta}$, for every $\xi, \eta \in \D$ (to simplify notations we adopt the symbol $\ip{\cdot}{\cdot}$ for both of
 them) and also that the embedding map $I_{\D,\D^\times}:\D\to \D^\times$ can be taken to act on $\D$ as $I_{\D,\D^\times}f=f$ for every $f \in \D$.

 Let now $\D[t] \subset \H \subset \D^\times[t^\times]$ be a rigged
Hilbert space,
 and let $\LDD$ denote the vector space of all continuous linear maps from $\D[t]$ into  $\D^\times[t^\times]$. {If $\D[t]$ is barreled (e.g., reflexive)}, an involution $X \mapsto X^\dag$ can be introduced in $\LDD$  by the equality
$$ \langle X^\dag \eta| \xi\rangle = \overline{\ip{X\xi}{\eta}}, \quad \forall \xi, \eta \in \D.$$  Hence, in this case, $\LDD$ is a $^\dagger$-invariant vector space.

If $\D[t]$ is a {smooth} space (e.g., Fr\'echet and reflexive), then $\LDD{}$ is a quasi *-algebra over $\LD$ \cite[Definition 2.1.9]{ait_book}.

We also denote by ${\mc L}(\D)$ the algebra of all continuous linear operators $Y:\D[t]\to \D[t]$ and by ${\mc L}(\D^\times)$ the algebra of all continuous linear operators $Z:\D^\times[t^\times]\to \D^\times[t^\times]$.
If $\D[t]$ is reflexive, for every $Y \in {\mc L}(\D)$ there exists a unique operator $Y^\times\in {\mc L}(\D^\times)$, the adjoint of $Y$, such that
$$ \ip{\Phi}{Yg} = \ip{Y^\times \Phi}{g}, \quad\forall \Phi \in \D^\times, g \in \D.$$ In a similar way, an operator $Z\in {\mc L}(\D^\times)$ has an adjoint $Z^\times\in {\mc L}(\D)$ such that $(Z^\times)^\times=Z$.

In the sequel we will need the following
\begin{lemma} \label{lemma_inverse} Let $\D[t] \subset \H \subset \D^\times[t^\times]$ be a rigged Hilbert space with $\D[t]$ Fr\'echet and reflexive. Let $S\in \LDD$ with $S=S^\dag$

The following statements are equivalent.
\begin{itemize}
\item[(i)] There exists $R\in {\mc L}(\D)$ such that $SR=R^\times S= I_{\D,\D^\times}$.
\item[(ii)]For every bounded subset $\M$ of $\D[t]$, there exist $A>0$ and a bounded subset $\N$ of $\D[t]$ such that
\begin{equation}\label{eqn_lemma} A^{1/2} \sup_{g\in \M}|\ip{f}{g}| \leq \sup_{h\in \N}|\ip{Sf}{h}| , \quad \forall f\in\D.\end{equation}
\end{itemize}
\end{lemma}
\begin{proof} (i)$\Rightarrow$(ii): We have, using the continuity of $R^\times$ from $\D^\times[t^\times]$ into itself,
$$ \sup_{g\in \M}|\ip{f}{g}|=\sup_{g\in \M}|\ip{SRf}{g}|=\sup_{g\in \M}|\ip{R^\times Sf}{g}|\leq c \sup_{h\in \N}|\ip{Sf}{h}|, \; \forall f\in \D,$$
for some $c>0$ and for some bounded subset $\N$ of $\D[t]$.

(ii)$\Rightarrow$(i): First we notice that \eqref{eqn_lemma} easily implies that $S$ is injective and since $S=S^\dag$ it has dense image in $\D^\times$. Hence $S^{-1}: {\sf Ran}\, S \to\D$ is well defined and, by \eqref{eqn_lemma},
$$A^{1/2} \sup_{g\in \M}|\ip{S^{-1}\Phi}{g}| \leq \sup_{h\in \N}|\ip{\Phi}{h}| , \quad \forall \Phi\in {\sf Ran}\, S.$$
Thus $S^{-1}: {\sf Ran}\, S [t^\times] \to\D [t^\times]$ is continuous and has a unique continuous extension $\widetilde{S^{-1}}$ from $\D^\times \to \D^\times$. We put $R:= \left(\widetilde{S^{-1}}\right)^\times$. Then $R$ satisfies the requirements in (i).
{Indeed, for every $f,g\in \D$,
\begin{align*}
\ip{SRf}{g} &= \ip{Rf}{Sg}=\ip{\widetilde{S^{-1}}^\times f}{Sg}\\
&=\ip{f}{{\widetilde{S^{-1}}}Sg}=\ip{f}{g},
\end{align*}
due to the fact that ${\widetilde{S^{-1}}}$ is an extension of $S^{-1}$. The proof that $R^\times S f=f$ is similar.}
\end{proof}

The notion of (continuous) frames in Hilbert spaces and its link with the theory of coherent states are well-known in the literature (see, e.g. \cite{AAG_book, AAG_paper, keiser}).

Let $(X,\mu)$ be a measure space with $\mu$ a $\sigma$-finite positive measure. A weakly measurable function $\zeta: x \in X \to \zeta_x\in\H$ is  a continuous frame in Hilbert space $\H$  if there exists $A,B>0$ such that:
\begin{equation}A\|f\|^2 \leq \int_X |\ip{f}{\zeta_x}|^2d\mu \leq B\|f\|^2, \quad \forall f \in \H.\end{equation}

Let $\D[t]\subset \H\subset \D^\times [t^\times]$ be a rigged Hilbert space and $x\in X\to \omega_x \in \D^\times$ a weakly measurable map; i.e. we suppose that, for every $f \in \D$, the complex valued function $x \mapsto \ip{f}{\omega_x}$ is $\mu$-measurable.
Since the form which puts $\D$ and $\D^\times$ in conjugate duality is an extension of the inner product of $\H$, we write $\ip{f}{\omega_x}$ for $\overline{\ip{\omega_x}{f}}$,  $f \in \D$.

\bedefi

\label{tandg}
 Let $\D[t]$ be a locally convex space, $\D^\times$ its conjugate dual and $\omega: x\in X\to \omega_x \in \D^\times$ a weakly measurable map, then:
\begin{itemize}
\item[i)] $\omega$ is \textit{total} if, $f \in \D $ and $\ip{f}{\omega_x}=0$ for almost every $x \in X$ implies $f=0$;
\item[ii)]$\omega$ is \textit{$\mu$-independent} if the unique measurable function $\xi$ on $\RN$ such that \\$\int_X \xi(x)\ip{g}{\omega_x} d\mu=0$, for every $g \in \D$, is $\xi(x)=0$ $\mu$-a.e.
\end{itemize}
\findefi

\bedefi
\label{basisf}
Given a rigged Hilbert space $\D\hookrightarrow\hh\hookrightarrow\D^\times$, a weakly measurable function $\omega: X\to \D^\times$ is called a {\em  distribution basis} for $\D$
 if, for every $f \in \D$, there exists a \emph{unique} measurable function $\xi_f$ such that:
$$\ip{f}{g}=\int_X \xi_f(x)\ip{\omega_x}{g} d\mu, \quad\forall g \in \D $$
and, for every $x \in  X$, the linear functional $f\in \D \to \xi_f(x)\in {\mb C}$ is continuous on $\D[t]$.
\findefi

\berem
If $\omega$ is a distribution basis, then it is $\mu$-independent. However a $\mu$-independent function $\omega$, need not be, in general, a distribution basis, but only uniqueness of the coefficient function $\xi_f(x)$ is guaranteed.
\enrem

If $\omega$ is a distribution basis, there exists a \textit{unique} weakly measurable map $\theta:X \to \D^\times$ such that $\xi_f(x)=\ip{f}{\theta_x}$, for every $f \in \D$. Hence the following identity holds:
\begin{equation}\label{schauder}
\ip{f}{g}=\int_X\ip{f}{\theta_x}\ip{\omega_x}{g} d\mu, \quad \forall f,g \in \D.
\end{equation}
We will call $\theta$ the conjugate (or dual) map of $\omega$. If $\theta$ is $\mu$-independent too, the roles of $\theta$ and $\omega$ in \eqref{schauder} are totally symmetric; hence,

\begin{prop}\label{ft}Let $\omega$ be a {distribution basis} and $\theta$ its conjugate map. If $\theta$ is $\mu$-independent, then  $\theta$ is also a distribution basis.
\end{prop}

\section{Distribution frames and Gel'fand bases}\label{sect_3}
\subsection{Bessel distribution maps and distribution frames}

Let $\omega:x\in X\to \omega_x \in\D^\times$ be a weakly measurable function. We suppose that
the sesquilinear form
\begin{equation}\label{eqn_omega} \Omega(f,g)= \int_X \ip{f}{\omega_x}\ip{\omega_x}{g}d\mu\end{equation} is well defined on $\D\times \D$. Then clearly, for every $f\in \D$, $\int_X |\ip{f}{\omega_x}|^2 d\mu <\infty$.

\begin{prop}\label{prop2} Let $\D[t]\subset\H\subset \D^\times[t^\times]$ be a rigged Hilbert space, with $\D[t]$ a Fr\'{e}chet space, and $\omega: x\in X \to \omega_x\in \D^\times$ a weakly measurable function. The following statements are equivalent.
\begin{itemize}
\item[(i)] For every $f \in \D$,
$$ \int_X |\ip{f}{\omega_x}|^2d\mu<\infty.$$
\item[(ii)]there exists a continuous seminorm $p$ on $\D[t]$ such that:
\begin{equation}\label{eqn_bessel1}\left( \int_X |\ip{f}{\omega_x}|^2d\mu\right)^{1/2}\leq p(f), \quad \forall f \in \D.\end{equation}
 \item[(iii)] for every bounded subset $\mathcal M$ of $\D$ there exists $C_{\mathcal M}>0$ such that:
\begin{equation}
\label{disbessel}
\sup_{f\in\mathcal M}{\Bigr |}\int_X\xi(x)\ip{\omega_x}{f}d\mu{\Bigl |}\leq C_{\mathcal M}\|\xi\|_2, \quad \forall \xi\in L^2(X,\mu).
\end{equation}
  \end{itemize}
\end{prop}
\begin{proof} (i)$\Rightarrow$(ii):
Consider the map
$$f \in \D \to \ip{f}{\omega_x}\in L^2(X,\mu).$$
This map is closable. Indeed, if $\{f_n\}$ is a sequence in $\D$ converging to $0$ in the $t$-topology, then $\{f_n\}$ converges weakly to $0$; i.e., $\ip{f_n}{F} \to 0$, for every $F\in \D^\times$. Thus if $\ip{f_n}{\omega_x} \to h(x)\in L^2(X, \mu)$, then  necessarily $h(x)=0$ $\mu$-a.e. in $X$. The statement follows by applying the closed graph theorem.

(ii)$\Rightarrow$(iii): Using \eqref{eqn_bessel1}, we get, for every $\xi\in L^2(X, \mu)$
\begin{equation}\label{eqn_8p} {\Bigr |}\int_X\xi(x)\ip{\omega_x}{f}d\mu{\Bigl |}\leq \|\xi\|_2 p(f).\end{equation}
Then, if $\mathcal M$ is a  bounded subset  of $\D$, putting $C_{\mathcal M}:=\sup_{f\in\mathcal M}p(f)$ we get the desired inequality.

(iii)$\Rightarrow$(i): From (\ref{disbessel}) we get
$$
{\Bigr |}\int_X\xi(x)\ip{\omega_x}{f}d\mu{\Bigl |}\leq C_{\mathcal M}\|\xi\|_2\quad \forall f\in\mathcal M,\quad \forall\xi(x)\in L^2(X, \mu).
$$
Then it follows that  $\ip{\omega_x}{f}\in L^2(X,\mu)$ \cite[Ch. 6, Ex. 4]{rudin}.
\end{proof}

\bedefi Let $\D[t]\subset\H\subset \D^\times[t^\times]$ be a rigged Hilbert space, with $\D[t]$ a Fr\'{e}chet space. If any of the equivalent conditions of Proposition \ref{prop2} is satisfied, we say that $\omega$ is a {\em Bessel distribution map}.

In particular, a Bessel distribution map $\omega$ is called {\em bounded} if there exists $B>0$ such that
\begin{equation}\label{eqn_bess_ex} \int_X|\ip{f}{\omega_x}|^2d\mu \leq B \|f\|^2,\; \forall f\in \D.\end{equation}
\findefi
\beex \label{ex_33} If a measurable function $\omega$ satisfies the condition
\eqref{eqn_bess_ex} then it is automatically a Bessel distribution map. However, a Bessel distribution map $\omega$ need not be necessarily bounded. Indeed, let us consider the rigged Hilbert space $\SC \subset L^2(\RN) \subset \SSS$ and  the distribution function $\omega$ defined by $\omega_x =\frac{1}{\sqrt{2\pi}}(1+y^2) e^{ixy}$, $x,y\in \RN$.
An easy computation shows that, for every $f \in \SC$, $\ip{f}{\omega_x}= (\hat f - D_x^2 \hat f)(x)$, $x\in \RN$, where $\hat{}$ denotes the Fourier transform and $D_x$ the derivative operator with respect to $x$.
Hence,
$$ \int_\RN|\ip{f}{\omega_x}|^2dx = \int_\RN|(\hat f - D_x^2 \hat f)(x)|^2dx.$$
The right hand side is the square of the norm of $\hat f$ in the Sobolev space $W^{2,2}(\RN)$ \cite[Vol. II, Section IX.6]{reed1}. This norm is continuous in $\SC$ with its usual topology, and it is well known that this norm is not equivalent to the $L^2$-norm.
\enex

\begin{lemma}\label{lemma2} Let $\D[t]\subset\H\subset \D^\times[t^\times]$ be a rigged Hilbert space, with $\D[t]$ a Fr\'{e}chet space. If $\omega$ is a Bessel distribution map then, for every $f \in \D$, the integral:
$$ \int_X\ip{f}{\omega_x}\omega_xd\mu$$ converges to an element of $\D^\times$.

Moreover, the map $f\in \D \to \int_X\ip{f}{\omega_x}\omega_xd\mu \in \D^\times$ is continuous.
\end{lemma}
\begin{proof}
 For every $f, g\in \D$, we have, by Proposition \ref{prop2},
\begin{equation}\label{eqn_bessel}\left|\int_X \ip{f}{\omega_x}\ip{\omega_x}{g}d\mu \right|\leq \|\ip{f}{\omega_x}\|_2 \|\ip{g}{\omega_x}\|_2 \leq p(f)p(g).\end{equation}
Hence $ \int_X\ip{f}{\omega_x}\omega_xd\mu \in \D^\times$.

The continuity of the map $f\in \D \to \int_X\ip{f}{\omega_x}\omega_xd\mu \in \D^\times$ follows by taking the sup of the previous inequality when $g$ runs over a bounded subset $\M$ of $\D$.
\end{proof}

{
Let $\omega$ be a Bessel distribution map and $\Omega$ the sesquilinear form defined in \eqref{eqn_omega}.
By \eqref{eqn_bessel}, we get
$$ |\Omega(f,g)|= \left| \int_X \ip{f}{\omega_x}\ip{g}{\omega_x}d\mu\right|\leq p(f)p(g), \quad \forall f,g \in \D$$
for some continuous seminorm $p$ on $\D[t]$.}
This means that $\Omega$ is jointly continuous on $\D[t]$. Hence there exists an operator $S_\omega \in \LDD$, with $S_\omega =S_\omega^\dag$, $S_\omega\geq 0$, such that
\begin{equation} \label{eqn_frameop}\Omega(f,g)=\ip{S_\omega f}{g}=\int_X \ip{f}{\omega_x}\ip{\omega_x}{g} d\mu, \quad \forall f,g\in \D\end{equation}
that is,
$$S_\omega f= \int_X \ip{f}{\omega_x}\omega_x d\mu, \quad \forall f \in \D.$$
If $\omega$ is a Bessel distribution map and $\xi\in L^2(X, \mu)$, we put
\begin{equation}\label{eqn_lambda_om}\Lambda_\omega^\xi(g):= \int_X \xi(x)\ip{\omega_x}{g}d\mu .\end{equation}
Then $\Lambda_\omega^\xi$ is a continuous conjugate linear functional on $\D$.
Hence, there exists a unique $\Phi^\xi_\omega\in \D^\times$ such that
$$ \ip{\Phi^\xi_\omega}{g}=\int_X\xi(x)\ip{\omega_x}{g}d\mu, \quad  \forall g\in\D.$$

Therefore we can define a linear map $T_\omega:L^2(X, \mu)\to \D^\times$, which will be called the {\em synthesis operator}, by
$$ T_\omega \xi = {\Phi^\xi_\omega}, \quad \xi \in L^2(X, \mu).$$
By \eqref{disbessel}, it follows that $T_\omega$ is continuous from $L^2(X, \mu)$, endowed with its natural norm, into $\D^\times[t^\times]$. Hence, it possesses a continuous adjoint $T_\omega^\times: \D[t]\to L^2(X, \mu)$, which is called the {\em analysis operator}, acting as follows:
$$T_\omega^\times: f\in \D \to \xi_f\in L^2(X, \mu), \mbox{ where } \xi_f(x)=\ip{f}{\omega_x}, \; x\in X.$$
As it is readily checked $S_\omega=T_\omega T_\omega^\times.$

We prove the following:
\begin{cor} \label{cor 5.6}
Assume that $\D[t]$ is a Fr\'echet space. If the weakly measurable map $\omega: x\in X\to \omega_x \in \D^\times$ is a Bessel distribution map
then for every continuous frame $\zeta$ on $\H$   there exists a continuous operator $Q:\hh\rightarrow \D^\times$ such that
 $Q f=\int_X\ip{f}{\zeta_x}\omega_xd\mu$, for every $f\in \H$.
\end{cor}
\begin{proof}
If $\omega$ is a Bessel distribution map, by (iii) of Proposition \eqref{prop2} the operator $T_\omega: L^2(X,\mu)\to \D^\times$ defined by $T_\omega\xi:=\int_X\xi(x){\omega_x}d\mu$ is continuous. On the other hand, since $\zeta$ is a frame, the analysis operator $T_\zeta^\times:\hh\rightarrow L^2(X,\mu)$ defined by $T_\zeta^\times f=\ip{f}{\zeta_x}$ is continuous. The statement follows by defining $Q=T_\omega T_\zeta^\times$.
\end{proof}

\subsection{Bounded Bessel distribution maps}
{Let us now consider the case when $\omega$ is a bounded Bessel distribution map; i.e., \eqref{eqn_bess_ex} holds. Then, if $\xi\in L^2(X,\mu)$, $\Lambda_\omega^\xi$
is bounded in $\D[\|\cdot\|]$; hence, it has
 bounded extension $\tilde{\Lambda}_\omega^\xi$ to $\H$, defined, as usual, by a limiting procedure.

Therefore, there exists a unique vector $h_\xi\in \H$ such that
$$\tilde{\Lambda}_\omega^\xi (g)= \ip{h_\xi}{g}, \quad \forall g\in\H.$$
This implies that the synthesis operator $T_\omega$ takes values in $\H$, it is bounded and $\|T_\omega\|\leq B^{1/2}$; its hilbertian adjoint $D_\omega:=T_\omega^*$ extends the analysis operator $T_\omega^\times$.

 The action of $D_\omega$ can easily be described: if $g\in \H$ and  $\{g_n\}$ is a sequence of elements of $\D$, norm converging to $g$, then the sequence $\{\eta_n\}$, where $\eta_n(x)=\ip{g_n}{\omega_x}$, is convergent in $L^2(X, \mu)$. Put $\eta=\lim_{n\to \infty}\eta_n$. Then,
 \begin{equation}
\ip{T_\omega \xi}{g} =\lim_{n\to \infty}\int_X \xi(x)\ip{\omega_x}{g_n}d\mu=\int_X\xi(x)\overline{\eta(x)}d\mu.
\end{equation} Hence $T_\omega^*g=\eta.$

{The function $\eta\in L^2(X,\mu)$ depends linearly on $g$, for each $x\in X$. Thus we can define a linear functional $\check{\omega}_x$ by
\begin{equation}\label{eqn_check}\ip{g}{\check{\omega}_x}= \lim_{n\to \infty}\ip{g_n}{\omega_x}, \quad g\in \H;\, g_n\to g.\end{equation}
Of course, for each $x\in X$, $\check{\omega}_x$ extends $\omega_x$;
however $\check{\omega}_x$ need not be continuous, as a functional on $\H$.
We conclude that $T_\omega^*$ associates to each $g\in \H$ the coefficient function $\ip{g}{\check{\omega}_x}\in L^2(X, \mu).$}

Moreover, in this case, the sesquilinear form
$ \Omega$ in \eqref{eqn_frameop}, which is well defined on $\D\times \D$, is bounded with respect to $\|\cdot\|$ and possesses a bounded extension $\hat{\Omega}$ to $\H$.
Hence there exists a bounded operator $\hat{S}_\omega$ in $\H$, such that
\begin{equation} \label{eqn_ext_Omega} \hat{\Omega}(f,g) =\ip{\hat{S}_\omega f}{g}, \quad \forall f,g \in \H. \end{equation}
Since
\begin{equation} \label{eqn_frameop_1}\ip{\hat{S}_\omega f}{g}=\int_X \ip{f}{\omega_x}\ip{\omega_x}{g} d\mu, \quad \forall f,g\in \D,\end{equation}
 $\hat{S}_\omega$ extends the {frame operator} $S_\omega$ and $S_\omega:\D\to \H$.} It is easily seen that
$\hat{S}_\omega = \hat{S}_\omega^*$ and $\hat{S}_\omega=T_\omega T_\omega^*$.

{
\bedefi \label{defn_distribframe}  Let $\D[t]\subset\H\subset \D^\times[t^\times]$ be a rigged Hilbert space, with $\D[t]$ a reflexive Fr\`{e}chet space and $\omega$ a Bessel distribution map.
We say that $\omega$ is a {\em  distribution frame} if there exist $A,B>0$ such that
\begin{equation} \label{eqn_frame_main1} A\|f\|^2 \leq \int_X|\ip{f}{\omega_x}|^2d\mu \leq B \|f\|^2, \quad \forall f\in \D. \end{equation}
\findefi
}
A distribution frame $\omega$ is clearly a bounded Bessel map. Thus, we can consider the operator $\hat{S}_\omega$ defined in \eqref{eqn_ext_Omega}. It is easily seen that, in this case,
$$ A\|f\| \leq \| \hat{S}_\omega f\| \leq B\|f\|,\quad \forall f\in \H. $$ This inequality, together with the fact that $\hat{S}_\omega$ is symmetric, implies that $\hat{S}_\omega$ has a bounded inverse $\hat{S}_\omega^{-1}$ everywhere defined in $\H$.

\berem
It is worth noticing that the fact that $\Omega$ and $S_\omega$ extend to $\H$ does not mean that $\omega$ is a frame in the Hilbert space $\H$, because we do not know if the extension of ${S}_\omega$ has the form of \eqref{eqn_frameop} with $f,g \in \H$.
\enrem
\begin{lemma}\label{lemma_38}Let $\omega$ be a distribution frame. Then, there exists $R_\omega\in {\mc L}(\D)$ such that $S_\omega R_\omega f=R_\omega^\times S_\omega f= f$, for every $f\in \D$.
\end{lemma}
\begin{proof} Taking into account that the topology induced on $\D$ by the topology $t^\times$ of $\D^\times$ is coarser than that induced by the norm of $\H$, if $\M$ is bounded in $\D[t]$, we get, for some $C>0$,
$$ \sup_{g\in \M}|\ip{f}{g}|^2 \leq C\|f\|^2 \leq \frac{C}{A} \int_X|\ip{f}{\omega_x}|^2d\mu \leq  \frac{BC}{A} \|f\|^2.$$
Let us define $h=\frac{f}{\left(\int_X|\ip{f}{\omega_x}|^2 d\mu\right)^{1/2}}$ and $\N= \{h\}$, which is obviously a bounded set in $\D[t]$. Then
$$\int_X|\ip{f}{\omega_x}|^2d\mu = \sup_{h\in \N} \left|\int_X \ip{f}{\omega_x}\ip{\omega_x}{h}d\mu\right|^2.$$
Hence
$$ \sup_{g\in \M}|\ip{f}{g}| \leq \frac{C}{A} \sup_{h\in \N} \left|\int_X \ip{f}{\omega_x}\ip{\omega_x}{h}d\mu\right|.$$
By Lemma \ref{lemma_inverse}, there exists $R_\omega \in {\mc L}(\D)$ such that
$S_\omega R_\omega f= R_\omega^\times S_\omega f= f$, for every $f \in \D$.
\end{proof}

\berem The operator $R_\omega$ acts as an inverse of $S_\omega$. But the operator $\hat{S}_\omega$ has a bounded inverse $\hat{S}_\omega^{-1}$ everywhere defined in $\H$: How do $R_\omega$ and $\hat{S}_\omega^{-1}$ compare? As we have seen ${\sf Ran}\,S_\omega \subset \H$. Hence $\hat{S}_\omega^{-1}S_\omega f=f$ for every $f\in \D$. On the other hand, $R_\omega^\times S_\omega f=f$ for every $f\in \D$.
Hence $R_\omega^\times h= \hat{S}_\omega^{-1}h$ for every $h\in {\sf Ran}\,S_\omega$.  Moreover $ \hat{S}_\omega \hat{S}_\omega^{-1}f=f$ for every $f\in \D$. Using the fact that  $S_\omega R_\omega f=f$ for every $f\in \D$, we obtain, for $f\in \D$,  $ \hat{S}_\omega \hat{S}_\omega^{-1}S_\omega R_\omega f=S_\omega R_\omega f$. Hence $\hat{S}_\omega \hat{S}_\omega^{-1} f=\hat{S}_\omega R_\omega f$, which implies that $\hat{S}_\omega^{-1} f=R_\omega f$, for every $f\in \D$. Thus, $\hat{S}_\omega^{-1} $ maps $\D$ into $\D$ and $R_\omega=\hat{S}_\omega^{-1}\upharpoonright_\D$.

\enrem

\begin{prop} \label{prop_dual2} Let $\omega$ be a distribution frame. Then there exists a measurable function $\theta$ such that
$$ \ip{f}{g}= \int_X\ip{f}{\theta_x}\ip{\omega_x}{g}d\mu, \quad \forall f,g \in \D.$$
\end{prop}
\begin{proof} By \eqref{eqn_frameop} and using Lemma \ref{lemma_38}, we get
$$ \ip{f}{g}= \ip{S_\omega R_\omega f}{g}= \int_X\ip{R_\omega f}{\omega_x}\ip{\omega_x}{g}d\mu =\int_X\ip{f}{R_\omega^\times\omega_x}\ip{\omega_x}{g}d\mu, \quad \forall f,g \in \D.$$

Then, if we put $\theta_x=R_\omega^\times\omega_x$, $x\in X$,  we get the equality
$$ \ip{f}{g}= \int_X\ip{f}{\theta_x}\ip{\omega_x}{g}d\mu, \quad \forall f,g \in \D.$$
\end{proof}

This equality shows that every $f\in \D$ can be expanded (in weak sense) in terms of $\omega$, as it happens for usual frames in Hilbert space.
The function $\theta: x\in X \to \theta_x\in \D^\times$ is clearly weakly measurable. It is in fact a Bessel map since
$$ \int_X |\ip{f}{\theta_x}|^2d\mu = \int_X |\ip{f}{R_\omega^\times \omega_x}|^2d\mu = \int_X |\ip{R_\omega f}{\omega_x}|^2d\mu <\infty.$$
Then the frame operator $S_\theta$ for $\theta$ is well defined and we have, for every $f,g \in \D$,
\begin{align*}\ip{S_\theta f}{g}&= \int_X\ip{f}{\theta_x} \ip{\theta_x}{ g}d\mu = \int_X\ip{ f}{R_\omega^\times\omega_x} \ip{R_\omega^\times \omega_x}{ g}d\mu \\
&= \int_X\ip{R_\omega f}{\omega_x} \ip{\omega_x}{ R_\omega g}d\mu
=\ip{S_\omega R_\omega f}{R_\omega g}\\ &= \ip{R_\omega^\times S_\omega R_\omega f}{ g} =\ip{I_{\D,\D^\times}R_\omega f}{g}.
\end{align*}
Hence $S_\theta= I_{\D,\D^\times}R_\omega$.

{Let us now show that the distribution function  $\theta$, constructed in Proposition \ref{prop_dual2}, is also a distribution frame, called the {\em canonical dual frame} of $\omega$}. Indeed, we have
$$ A\|R_\omega f\|^2 = A\ip{R_\omega f }{R_\omega f} \leq \ip{S_\omega R_\omega f}{R_\omega f}= \ip{f}{R_\omega f}\leq \|f\|\|R_\omega f\|.$$
Hence $A\|R_\omega f\|\leq \|f\|$ and
$$ \ip{S_\theta f }{ f} =\ip{I_{\D,\D^\times}R_\omega f}{f}=\ip{R_\omega f}{f}\leq \|R_\omega f\|\|f\| \leq A^{-1}\|f\|^2.$$
Moreover,
\begin{align*}\|f\|^4 &=\ip{R_\omega^\times S_\omega f}{f}^2 = \ip{S_\omega f}{R_\omega f}^2 \\ &\leq \ip{S_\omega f}{f} \ip{S_\omega R_\omega f}{R_\omega f}\\ & \leq B\|f\|^2\ip{f}{R_\omega f}
=B\|f\|^2\ip{S_\theta f}{ f} \end{align*}
In conclusion,
$$ B^{-1}\|f\|^2 \leq \ip{S_\theta f}{ f} \leq A^{-1}\|f\|^2, \quad \forall f\in \D$$
and therefore $\theta$ is a distribution frame.

\medskip
We conclude this section with the following
\begin{lemma}\label{lemmanew} Let $(X,\mu)$ be a measure space with $\mu$ a $\sigma$-finite measure. Suppose that $\omega$ is a measurable distribution map for which there exists $A,B>0$ such that
$$ A\|\varphi\|_2^2 \leq \left\| \int_X \varphi(x) \omega_x d\mu\right\|^2\leq B \|\varphi\|_2^2, \quad \forall \varphi \in L^2(X,\mu).$$
If $\xi$ is a measurable function such that $\int_X \xi(x) \omega_x d\mu$ exists in $\H$, then $\xi \in L^2(X,\mu)$.
\end{lemma}
\begin{proof} The argument used in the proof of (iii)$\Rightarrow$(i) of Proposition \ref{prop2} shows that $\omega$ is a Bessel distribution map. Hence, $\int_X \varphi(x) \omega_x d\mu\in \H$, for every $\varphi \in L^2(X,\mu)$.

Without loss of generality, we assume that $\xi$ is real valued and nonnegative since every complex function is a linear combination of four nonnegative ones. Let us first suppose that $\xi$ is zero outside of a set of finite measure $Y$. Then there exists a sequence $0\leq s_i\leq s_2 \leq \cdots \leq s_n\leq \cdots \leq \xi$ of simple functions such that $s_n \to \xi$ pointwise in $X$ and $s_n(x)=0$ for $x\in X\setminus Y$.
Since $s_n(x)|\ip{\omega_x}{f}| \leq \xi(x)|\ip{\omega_x}{f}|$, by the dominated convergence theorem,
$$\int_X \xi(x)\ip{\omega_x}{f}d\mu =\lim_{n\to \infty} \int_X s_n(x)\ip{\omega_x}{f}d\mu, \quad \forall f\in\D.$$
Clearly, $s_n\in L^2(X, \mu)$. Then, for every $n \in {\mb N}$,  the conjugate linear functional $G_n(f)=\int_Xs_n(x)\ip{\omega_x}{f}d\mu$, $f\in \D$, is bounded on $\D$, and $G_n(f) \to \int_X\xi(x)\ip{\omega_x}{f}d\mu$. The uniform boundedness principle (applied to the continuous extensions of the $G_n$'s) implies that
$G(f):=\int_X\xi(x)\ip{\omega_x}{f}d\mu$ is also a bounded conjugate linear functional on $\D$. Hence, $\int_X\xi(x){\omega_x}d\mu \in \H$ and
$$\int_X\xi(x){\omega_x}d\mu=\lim_{n\to\infty} \int_Xs_n(x){\omega_x}d\mu.$$
Then, by the assumption,
$$C\|s_n-s_m\|_2 \leq \left\|\int_X (s_n(x)-s_m(x)){\omega_x}d\mu \right\| \to 0.$$
This implies that $\xi \in L^2(X,\mu)$.

If $\xi$ is a generic nonnegative function, one proceeds in similar way starting with defining, for each $n\in {\mb N}$, $\xi_n(x) = \xi(x)\chi_{{\scriptscriptstyle{X_n}}}(x)$, where $\{X_n\}$ is a sequence of sets of finite measure such that $X_n \subset X_{n+1}$, $X=\bigcup_{n\in {\mb N}}X_n$, and $\chi_{{\scriptscriptstyle{X_n}}}$ denotes the characteristic function of $X_n$.
\end{proof}
\begin{cor} \label{cor_ind}Let $\omega$ satisfy the assumptions of Lemma \ref{lemmanew}. If $\xi$ is a measurable function such that
 $\int_X \xi(x)\omega_x d\mu=0$, then $\xi=0$ $\mu$-a.e.
\end{cor}
\subsection{Parseval distribution frames}

\bedefi A weakly measurable distribution function $\omega$ is called a {\em Parseval distribution frame} if $$\int_X|\ip{f}{\omega_x}|^2d\mu = \|f\|^2, \quad f\in \D.$$
\findefi
It is clear that a Parseval distribution frame is a bounded Bessel distribution map and a frame in the sense of Definition \ref{defn_distribframe} with $S_\omega=I_\D$, the identity operator of $\D$.
\begin{lemma} \label{lemma_00}Let  $\D\subset\H\subset \D^\times$ be a rigged Hilbert space and $\omega: x\in X \to \omega_x\in \D^\times$ a weakly measurable map. The following statements are equivalent.
\begin{itemize}
\item[(i)] $\omega$ is a Parseval distribution frame;
\item[(ii)]$\ip{f}{g}= \int_X \ip{f}{\omega_x}\ip{\omega_x}{g}d\mu, \quad \forall f, g \in \D$;
\item[(iii)]$f = \int_X \ip{f}{\omega_x}\omega_x d\mu$,  the integral on the r.h.s. is understood as a continuous conjugate linear functional on $\D$, that is an element of $\D^\times$.
\end{itemize}
\end{lemma}
\begin{proof} (i)$\Rightarrow$ (ii):
If $\omega$ is a Parseval frame, then by \eqref{eqn_frameop} we obtain $\ip{S_\omega f}{f}=\|f\|^2$, for every $f \in \D$. Then  by the polarization identity
we easily get
$$ \ip{S_\omega f}{g}=\ip{f}{g}, \quad \forall f, g \in \D.$$

(ii)$\Rightarrow$ (iii): For every fixed $f\in\D$, we put $\xi_f(x)=\ip{f}{\omega_x}.$ Then $\xi_f\in L^2(X,\mu)$.
Therefore for every $f,g \in \D$ we have:
$$\ip{T_\omega \xi_f}{g}=\int_X \ip{f}{\omega_x}\ip{\omega_x}{g}d\mu=\ip{f}{g}.$$
Thus $$T_\omega \xi_f=f = \int_X \ip{f}{\omega_x}\omega_x d\mu.$$

(iii)$\Rightarrow$(i) In this situation by hypothesis, for every fixed $f\in\D$,  we have $T_\omega\xi_f=f=\int_X \ip{f}{\omega_x}\omega_x d\mu$
 and
$$\ip{T_\omega\xi_f}{f}=\int_X \ip{f}{\omega_x}\ip{\omega_x}{f}d\mu=\ip{f}{f}=\|{f}\|^2.$$
\end{proof}

The representation in (iii) of Lemma \ref{lemma_00} is not necessarily unique. 

\subsection{Gel'fand distribution bases}
As mentioned in the introduction, Gel'fand \cite[Ch.4, Theorem 2]{gelf3} called a $\D^\times$-valued function $\zeta$ a {\em complete system}, if it satisfies the Parseval equality and it has the property that every $f\in \D$ can be uniquely written as $f=\int_X\ip{f}{\zeta_x}\zeta_x d\mu$, in the weak sense.
As we shall see in the following discussion, these two conditions are a good substitute for the notion of an {\em orthonormal basis} which is meaningless in the present framework.

\begin{prop}\label{prop_gelfandbasis}Let  $\D\subset\H\subset \D^\times$ be a rigged Hilbert space and let $\zeta: x\in X \to \zeta_x\in \D^\times$ be a Bessel distribution map. Then the following statements are equivalent.
\begin{itemize}
\item[(a)] $\zeta$ is a $\mu$-independent Parseval distribution frame.
\item[(b)] The synthesis operator $T_\zeta$ is an isometry of $L^2(X,\mu)$ onto $\H$.
\end{itemize}
\end{prop}
\begin{proof}
 (a)$\Rightarrow$(b): Since $\zeta$ is $\mu$-independent, $T_\zeta$ is injective. If $f\in \H$ we put $\xi_f(x)=\ip{f}{\check{\zeta}_x}$. Then,
$$T_\zeta \xi_f=f = \int_X \ip{f}{\check{\zeta}_x}\zeta_x d\mu.$$
Hence, $T_\zeta$ is also onto. The isometry property follows immediately from the Parseval identity.

(b)$\Rightarrow$(a): This is an immediate consequence of $T_\zeta$ being an isometry and of Corollary \ref{cor_ind}.
\end{proof}

We will call a weakly measurable function $\zeta$ {\em Gel'fand distribution basis} if it satisfies one of the equivalent conditions of Proposition \ref{prop_gelfandbasis}.
{\begin{cor} \label{cor_gelfandbasis}Let $\zeta$ be a Gel'fand distribution basis. The following statements hold.
\begin{itemize}
\item[(i)] For every $f\in \H$ there exists a unique function $\xi_f\in L^2(X, \mu)$ such that
$$ f=\int_X \xi_f(x)\zeta_x d\mu.$$ In particular, if $f\in \D$, then $\xi_f(x)=\ip{f}{\zeta_x}$ $\mu$-a.e.
\item[(ii)]For every fixed $x\in X$, the map $f\in \H\to \xi_f(x)\in {\mb C}$ defines as in \eqref{eqn_check} a linear functional $\check{\zeta_x}$ on $\H$ and
$$ f=\int_X \ip{f}{\check{\zeta_x}}{\zeta_x} d\mu, \quad \forall f\in \H.$$
\end{itemize}
 \end{cor}  }
\beex \label{ex_3.15}The family of functions $\zeta_x(y)= \frac{1}{\sqrt{2\pi}}e^{-ixy}$, $x\in \RN$, considered in the introduction, is a Gel'fand  distribution basis. This can be deduced by applying standard results on the Fourier transform. In particular, denoting as usual by $\hat{g}$, $\check{g}$, respectively, the Fourier transform and the inverse Fourier transform of $g\in L^2(\RN)$, we have
\begin{align*} (T_\zeta \xi)(y) &= \frac{1}{\sqrt{2\pi}}\int_\RN\xi(x)e^{-ixy}dx=\hat{\xi}(y), \quad \forall \xi \in L^2(\RN);\\
T_\zeta^*f &=\check{f}, \quad \forall f \in L^2(\RN).
\end{align*}
\enex

\beex \label{ex_3.16}Let us consider the function $\delta: x\in \RN\to \delta_x\in \SSS$, where $\delta_x$ stands for the $\delta$ distribution centered at $x$. As it is known, $\delta$ acts in the following way
$\ip{\delta_x}{\phi}=\overline{\phi(x)} $, for every $\phi\in \SC$.
Then one trivially has
$$\int_\RN |\ip{\delta_x}{\phi}|^2dx= \int_\RN|\phi(x)|^2 dx =\|{\phi}\|^2_2,\quad \forall \phi\in \SC ;$$
hence, $\delta$ is a Parseval frame and it is $\mu$-independent.
If $\xi \in L^2(\RN)$, we have, for every $\phi\in \SC$,
$$\ip{T_\delta \xi}{\phi} =\int_{\RN}\xi(x)\ip{\delta_x}{\phi}dx=\int_{\RN}\xi(x)\overline{\phi(x)}dx =\ip{\xi}{\phi}.$$
Hence $T_\delta\xi=\xi$, for every $\xi \in L^2(\RN)$ and, clearly, $T_\delta^* f=f$, for every $f\in \H=L^2(\RN)$.
\enex
\medskip

Proposition \ref{prop_gelfandbasis} and Corollary \ref{cor_gelfandbasis} suggest a more general class of bases that will play the same role as Riesz bases in the ordinary Hilbert space framework.
\begin{prop}\label{prop_rieszbasis}Let  $\D\subset\H\subset \D^\times$ be a rigged Hilbert space and let $\omega: x\in X \to \omega_x\in \D^\times$ be a Bessel distribution map. Then the following statements are equivalent.
\begin{itemize}
\item[(a)]  $\omega$ is a $\mu$-independent distribution frame.
\item[(b)] If $\zeta$ is a Gel'fand distribution basis, then the operator $W$ defined, for $f\in \H$, by
$$ f=\int_X \xi_f(x)\zeta_x d\mu \to W f= \int_X \xi_f(x)\omega_x d\mu$$ is continuous and has a bounded inverse.
\item[(c)] The synthesis operator $T_\omega$ is a topological isomorphism of $L^2(X, \mu)$ onto $\H$.
\item[(d)] { $\omega$ is total and there exist $A,B>0$ such that
\begin{equation}\label{eqn_ref} A\|\xi\|_2^2 \leq \left\|\int_X \xi(x)\omega_x d\mu \right\|^2 \leq B\|\xi\|_2^2, \quad \forall \xi \in L^2(X, \mu).\end{equation}}
\end{itemize}
\end{prop}
\begin{proof} {(a)$\Rightarrow$(b): $W$ is clearly a well-defined linear operator. W is injective. Indeed, if $Wf=0$, by (a), $\xi_f=0$ $\mu$-a.e. This in turn implies that $f=0$.
Let now $g\in \H$. By (a) it follows that there exists a unique function $\xi_g \in L^2(X, \mu)$ such that $g=\int_X \xi_g(x)\omega_x d\mu$. Define $f=\int_X \xi_g(x)\zeta_x d\mu$. Then $Wf=g$. Hence, $W$ is surjective. Thus, $W$ is everywhere defined in $\H$ and has an everywhere defined inverse. We prove that $W$ is closed.
Let $\{f_n\}$ be a sequence in $\H$ such that $\|f_n\|\to 0$ and $Wf_n \to h\in \H$. Let $f_n=\int_X \xi_{f_n}(x)\zeta_x d\mu$, $\xi_{f_n} \in L^2(X, \mu)$. Since $\|f_n\|=\|\xi_{f_n}\|_2 \to 0$, we have, for every $g\in \D$,
$$|\ip{Wf_n}{g}|=\left|\int_X  \xi_{f_n} \ip{\omega_x}{g}d\mu\right|\leq \|\xi_{f_n}\|_2 \|\ip{\omega_x}{g}\|_2\to 0.
$$
Hence $\ip{h}{g}=0$, for every $g\in \D$ and, therefore, $h=0$.
By the closed graph theorem we conclude that $W$ is bounded.
 The fact that $W^{-1}$ is also bounded follows from the inverse mapping theorem.}

(b)$\Rightarrow$(c): We begin with proving that $T_\omega$ is an isomorphism. Suppose that $T_\omega \xi=0$ with $\xi \in L^2(X, \mu)$. Then,
$$W^{-1}T_\omega\xi = \int_X \xi(x) \zeta_x d\mu=0.$$
The $\mu$-independence of $\zeta$ implies that $\xi(x)=0$ $\mu$-a.e.

Let now $g$ be an arbitrary element of $\H$ and $f:=W^{-1}g$. Then, by Corollary \ref{cor_gelfandbasis}, there exists a unique $\xi_f \in L^2(X, \mu)$ such that $f=\int_X \xi_f (x) \zeta_x d\mu$. Then $Wf=\int_X \xi_f (x) \omega_x d\mu$. Hence $T_\omega \xi_f=g$. Thus $T_\omega$ is surjective.
It remains to prove that $T_\omega$ is bounded and has bounded inverse. Let $\xi\in L^2(X, \mu)$ and define $f:=T_\omega \xi$. Then, by the definition itself, it follows that $\|W^{-1}f\|=\|\xi\|_2$. Thus,
$$ \|T_\omega \xi\|=\|f\| =\|WW^{-1}f\|\leq \|W\|\|W^{-1}f\|= \|W\|\|\xi\|_2.$$
The fact that $T_\omega^{-1}$ is also bounded follows again from the inverse mapping theorem.

{(c)$\Rightarrow$(d): Since $T_\omega$ is bounded with bounded inverse, there exist $A,B>0$ such that \eqref{eqn_ref} holds. Moreover the analysis operator $T_\omega^*$ is also bounded with bounded inverse. Its injectivity then implies that $\omega$ is total.

(d)$\Rightarrow$(a): By \eqref{eqn_ref} it follows that $T_\omega$ is injective and has closed range. Then ${\sf Ran}T_\omega = ({\sf Ker} T_\omega^*)^\perp =\H$, since $\omega$ is total. Hence $T_\omega$ is bounded with bounded inverse. The same is then true for $T_\omega^*$, the analysis operator. Hence, there exist $A', B'>0$ such that
$$ A'\|f\|^2 \leq \|T_\omega^* f\|_2^2 \leq B'\|f\|^2,\quad \forall f\in \H.$$
If, in particular, $f\in \D$, then $T_\omega^*f=\xi_f$ with $\xi_f(x)= \ip{f}{\omega_x}$. Thus,
$$ A'\|f\|^2 \leq \int_X |\ip{f}{\omega_x} |^2 d\mu \leq B'\|f\|^2, \quad \forall f\in \D.$$
Hence $\omega$ is a distribution frame. The $\mu$-independence of $\omega$ follows immediately from \eqref{eqn_ref} and Corollary \ref{cor_ind}.}
\end{proof}

We will call {\em Riesz distribution basis} a weakly measurable function $\omega$ satisfying one of the equivalent conditions of Proposition \ref{prop_rieszbasis}.

\begin{prop} If $\omega$ is a Riesz distribution basis then $\omega$ possesses a unique dual frame $\theta$ which is also a Riesz distribution basis.
\end{prop}

\begin{proof} By Proposition \ref{prop_dual2}, there exists a canonical dual distribution frame $\theta$ and by the $\mu$-independence of $\omega$ the uniqueness of $\theta$
 follows. Let $W$ be the the operator in (b) of Proposition \ref{prop_rieszbasis}. The operator $W^*$ acts on $h\in \D$ as $W^*h=\int_X\ip{h}{\omega_x}\zeta_xd\mu$. Moreover, $W^*$ is invertible and has bounded inverse.
 Let $V$ be the following operator
 $$ f=\int_X \xi_f(x)\zeta_x d\mu \to V f= \int_X \xi_f(x)\theta_x d\mu.$$
 Then, for every $h\in\D$,
 $$VW^*h =V\int_X \ip{h}{\omega_x}\zeta_x d\mu = \int_X\ip{h}{\omega_x}\theta_x d\mu =h$$
 and, on the other hand, for every $g\in\D$
 $$ V^*Wg= \int_X \ip{g}{\zeta_x}\zeta_x d\mu=g.$$
 This proves that $V=(W^*)^{-1}$. Thus by (b) of Proposition \ref{prop_rieszbasis} it follows that $\theta$ is also a Riesz distribution basis.
 \end{proof}

\subsection{Transforming Gel'fand bases}
Proposition \ref{prop_rieszbasis} (b) characterizes a Riesz distribution basis in terms of the bounded operator $W$ acting on the Hilbert space $\H$. Let $W^*$ be its Hilbert adjoint
and suppose that $W^*\upharpoonright_\D$ maps $\D[t]$ into itself continuously. Then $W$ has an extension $\hat{W}$ to $\D^\times$ which is weakly continuous and $\hat{W}^\times \upharpoonright_\D= W^*\upharpoonright_\D$. In this case, we obtain
$$ \ip{f}{\hat{W}^\times g}= \int_X \ip{f}{\zeta_x}\ip{\zeta_x}{\hat{W}^\times g}d\mu = \int_X \ip{f}{\zeta_x}\ip{\hat{W}\zeta_x}{ g}d\mu. $$
On the other hand,
$$ \ip{Wf}{g}= \int_X \ip{f}{\zeta_x}\ip{\omega_x}{ g}d\mu. $$
Using the $\mu$-independence of $\zeta$ we get $\hat{W}\zeta_x =\omega_x$ $\mu$-a.e. Hence, in this case, $\omega$ is obtained by transforming $\zeta$ by means of the weakly continuous operator $\hat{W}$.
For this reason it is of some interest to consider distribution functions $\omega$ that are the image of $\zeta$ through an operator $M$ which maps $\D^\times$ into itself.

\medskip
Let $\zeta:x\in X \to \zeta_x\in \D^\times$ be a Gel'fand distribution basis and $M:\D^\times \to \D^\times$ a linear operator. Put $\omega_x = M\zeta_x$, $x\in X$.
If $M$ is weakly continuous then $\omega$ is weakly measurable and
$$ \int_X|\ip{f}{\omega_x}|^2d\mu = \int_X|\ip{f}{M\zeta_x}|^2d\mu=\int_X|\ip{M^\times f}{\zeta_x}|^2d\mu =\|M^\times f\|^2, \quad \forall f\in \D.$$
This shows that $\omega$ is a bounded Bessel map if and only if $M^\times$ is a bounded operator. This is not always the case as shown in Example \ref{ex_33}: the distribution function $\omega$ considered there is in fact the image of $\zeta_x= \frac{1}{\sqrt{2\pi}}e^{ixy}$ through the unbounded operator $M:\phi\in \SC \to M\phi\in \SC$ with $(M\phi)(y) = (1+y^2)\phi(y)$, $y\in \RN$.

Transforming an orthonormal basis by means of an operator $M$ can produce very different results.
\begin{prop} \label{prop_320} Let $\zeta:x\in X \to \zeta_x\in \D^\times$ be a Gel'fand distribution basis and $M:\D^\times \to \D^\times$ a weakly continuous map. Define a weakly measurable function $\omega$ by $\omega_x=M\zeta_x$, $x\in X$. Then $\omega$ is a Bessel distribution map and the following statements hold.
\begin{itemize}
\item[(i)] $\omega$ is a bounded Bessel distribution map if and only if $M^\times$ is bounded (with respect to the Hilbert norm).
\item[(ii)]$\omega$ is a distribution frame if and only if $M^\times$ is continuous from $\D[\|\cdot\|]$ into $\D[\|\cdot\|]$ and it has an inverse, continuous from ${\sf Ran}\, M^\times[\|\cdot\|]$ into $\D[\|\cdot\|]$.
\item[(iii)]$\omega$ is a Parseval frame if and only if $M^\times$ is isometric.
\item[(iv)] If $M$ has a continuous inverse, everywhere defined in $\D^\times$, then $\omega$ is a distribution basis.
\item[(v)]If $M^\times$ is bounded and has a bounded inverse everywhere defined on $\D$, then $\omega$ is a Riesz distribution basis.
\item[(vi)]If $M^\times$ is isometric and $M^\times \D=\D$, then $\omega$ is a Gel'fand distribution basis.

\end{itemize}
\end{prop}
\begin{proof} Since $M$ is weakly continuous $M^\times:\D \to \D$ exists and is weakly continuous. Then we have
\begin{equation}\label{eqn_trans}\int_X|\ip{f}{\omega_x}|^2d\mu =\int_X|\ip{f}{M\zeta_x}|^2d\mu =\int_X|\ip{M^\times f}{\zeta_x}|^2d\mu =\|M^\times f\|^2.\end{equation}
From this equality it follows immediately that $\omega$ is a Bessel distribution map and also the equivalence stated in (i).

As for (v), by \eqref{eqn_trans} it follows that $\omega$ is a distribution frame. We only need to prove that $\omega$ is $\mu$-independent. But this follows immediately from the $\mu$-independence of $\zeta$ and from the assumption $M^\times \D=\D$.

For proving (vi) we can apply Proposition \ref{prop_gelfandbasis} and notice as in (v) that $\omega$ is $\mu$-independent.

Now we prove (ii).
Assume that $\omega$ is a distribution frame. Then, there exist $A,B>0$ such that
$$ A\|f\|^2 \leq \int_X|\ip{f}{\omega_x}|^2d\mu \leq B \|f\|^2, \quad \forall f\in \D $$
or, equivalently
\begin{equation}\label{eqn_frame_4} A\|f\|^2 \leq \|M^\times f\|^2 \leq B \|f\|^2, \quad \forall f\in \D .\end{equation}
The previous inequalities show that $M^\times$ is continuous from $\D[\|\cdot\|]$ into $\D[\|\cdot\|]$ and injective. Moreover $(M^\times)^{-1}: {\sf Ran\, M^\times}[\|\cdot\|]\to \D[\|\cdot\|]$ is continuous.
Conversely, if $M^\times$ is continuous from $\D[\|\cdot\|]$ into $\D[\|\cdot\|]$ and it has an inverse, continuous from ${\sf Ran}\, M^\times[\|\cdot\|]$ into $\D[\|\cdot\|]$, it is clear that there exist $A,B>0$ such that
\eqref{eqn_frame_4} holds.

Let us now we prove (iii). If $M^\times$ is isometric, then for every $f\in \D$,
$$\|f\|^2=\|M^\times f\|^2= \int_X |\ip{f}{\omega_x}|^2d\mu.$$
Thus, $\omega$ is a Parseval distribution frame.

Conversely, if $\omega$ is a Parseval distribution frame, then, by \eqref{eqn_trans},  $\|f\|=\|M^\times f\|$, for every $f\in \D$.

Finally we prove (iv). Let us assume that $M$ has a continuous everywhere defined inverse, then by \cite[\S 38,n.4, (8)]{Kothe} $M^\times$ has a continuous inverse everywhere defined in $\D$ and $(M^\times)^{-1} =(M^{-1})^\times$.
Hence,  from the equality
$$ \ip{f}{g}=\int_X\ip{f}{\zeta_x} \ip{\zeta_x}{g}d\mu,\quad \forall f,g\in \D $$
we obtain
\begin{align*} \ip{f}{M^\times g}&=\int_X\ip{f}{\zeta_x} \ip{\zeta_x}{M^\times g} d\mu \\ &= \int_X\ip{f}{\zeta_x} \ip{M\zeta_x}{ g} d\mu
= \int_X\ip{f}{\zeta_x} \ip{\omega_x}{ g}d\mu.
\end{align*}
Hence $$\ip{M^{-1}f}{M^\times g} =\int_X\ip{M^{-1}f}{\zeta_x} \ip{\omega_x}{ g}d\mu.$$
Therefore,
$$\ip{f}{ g} =\int_X\ip{M^{-1}f}{\zeta_x} \ip{\omega_x}{ g}d\mu.$$
The continuity of the functional $f\to \ip{M^{-1}f}{\zeta_x}$ is easy.
\end{proof}

\berem As is known frames and Riesz bases in a Hilbert space can be characterized through the action of some linear operator $T$ on the elements of a orthonormal basis $\{e_n\}$. More precisely, a sequence $\{f_n\}$ is a frame if and only if $f_n=Te_n$ with $T$ bounded and surjective \cite[Theorem 5.4.4]{christensen} and $\{f_n\}$ is a Riesz basis if and only if $f_n=Te_n$ with $T$ bounded and bijective  \cite[Definition 3.3.1]{christensen}. It is natural to pose the question as to whether similar results hold for measurable distribution functions. Proposition \ref{prop_320} provides a partial answer to this question, since the converse of the statements (iv), (v), (vi) have not yet been established. The discussion before Proposition \ref{prop_320} about (b) of Proposition \ref{prop_rieszbasis} shows that the existence of an operator of ${\mc L}(\D^\times)$ transforming a Gel'fand distribution basis into a Riesz one requires additional assumptions. We leave this problem open.
\enrem

\section{Distribution frames and operators}\label{sect_4}
{In this section we will discuss the interplay of frames and operators in some specific situation. The operators considered in Example \ref{ex_gelf} and in Example \ref{ex_mult} are closely related with the continuous {\em frame multipliers} considered by Balasz {\em et al.} in \cite{balaszetal}. In both cases in fact the operators under consideration are determined by some sufficiently regular function which defines a multiplication operator.
Let us discuss some simple cases.}
\beex \label{ex_gelf}
The simplest operators one can define when having at hand an orthonormal basis $\{e_n\}$ in Hilbert spaces are {diagonal operators} with respect to $\{e_n\}$. Let us suppose that $\zeta: x\in X\to \zeta_x \in \D^\times$ is a Gel'fand distribution basis. Then, {\em diagonal} operators can be introduced as follows.

Let $\alpha$ be a (complex valued) measurable function such that $$\int_X |\alpha(x)\ip{f}{\check{\zeta}_x}|^2d\mu <\infty, \; \forall f\in \D.$$
We define a linear operator $A$ on $\D$ by
$$ Af= \int_X \alpha(x)\ip{f}{\zeta_x}\zeta_xd\mu, \quad  f\in \D.$$
If we put
$$ A^\dag g= \int_X \overline{\alpha(x)}\ip{g}{\zeta_x}\zeta_xd\mu, \quad  f\in \D$$
one can easily check that
$$ \ip{Af}{g}= \ip{f}{A^\dag g}, \quad \forall f,g\in \D.$$
Hence $A$ is a closable operator in $\H$.
The domain of its closure $\overline{A}$ is
$$ D(\overline{A})=\left\{ f \in \H: \int_X |\alpha(x)\ip{f}{\check{\zeta}_x}|^2d\mu <\infty\right\},$$
where $\check\zeta$ is understood in the sense of the approximation described by \eqref{eqn_check}.

The operator $A$ is bounded if and only if $\alpha \in L^\infty (X, \mu)$.
The spectrum $\sigma(\overline{A})$ is given by the closure of the essential range of $\alpha$ (see, for instance, \cite[p. 15]{davies}); that is, the set of $z \in {\mb C}$ such that
$$ \mu\{x: |\alpha(x)-z|<\epsilon\} >0, \quad \forall \epsilon >0.$$
Moreover, if $A\in \LD$,  for almost every $x\in X$, $\alpha(x)$ is a generalized  eigenvalue of $A$.
Indeed,  in this case we have, on the one hand,
\begin{equation}\label{eqn_expA}\ip{Af}{g} =\ip{f}{A^\dag g}=\int_X \ip{f}{\zeta_x}\ip{\zeta_x}{A^\dag g}d\mu= \int_X \ip{f}{\zeta_x}\ip{(A^\dag)^\times\zeta_x}{ g}d\mu, \; \forall f,g \in \D\,\end{equation}
and, on the other hand,
$$\ip{Af}{g}=\int_X \alpha(x)\ip{f}{\zeta_x}\ip{\zeta_x}{g}d\mu,\quad \forall f,g \in \D.$$
Hence, for almost every $x\in X$,
$$ \ip{(A^\dag)^\times\zeta_x}{g} = \alpha(x) \ip{\zeta_x}{g}, \quad\forall g \in \D.$$

\beex \label{ex_mult} In analogy with what has been done for Riesz bases in \cite{bit_jmp}, operators can be naturally defined also by taking a Riesz distribution basis $\omega$ and its dual $\theta$. Indeed, we define a linear operator $H$ in this way.
Let $\alpha$ be a (complex valued) measurable function such that $$\int_X |\alpha(x)\ip{f}{\check{\theta}_x}|^2d\mu <\infty, \; \forall f\in \D.$$
We define a linear operator $H$ on $\D$ by
$$ Hf= \int_X \alpha(x)\ip{f}{\theta_x}\omega_x d\mu. $$
If we put
$$ H^\dag g= \int_X \overline{\alpha(x)}\ip{g}{\omega_x}\theta_xd\mu, \quad  g\in \D,$$
then,
$$ \ip{Hf}{g}= \ip{f}{H^\dag g}, \quad \forall f,g\in \D.$$
Hence $H$ is a closable operator in $\H$.
The domain of its closure $\overline{H}$ is
$$ D(\overline{H})=\left\{ f \in \H: \int_X |\alpha(x)\ip{f}{\check{\theta}_x}|^2d\mu <\infty\right\}.$$
where  $\check{\theta}$  is understood in the sense of \eqref{eqn_check}.
As well as in the case of Example \ref{ex_gelf}, if $H\in \LD$, then,  for almost every $x\in X$, $\alpha(x)$ is a generalized  eigenvalue of $H$ with generalized eigenvector $\omega_x$.
Indeed, with computations very similar to those made in Example \ref{ex_gelf} one gets,
for almost every $x\in X$,
$$ \ip{(H^\dag)^\times\omega_x}{g} = \alpha(x) \ip{\omega_x}{g}, \quad\forall g \in \D.$$
Analogously one can prove that for almost every $x\in X$, $\overline{\alpha(x)}$ is a generalized  eigenvalue of $H^\dag$ with generalized eigenvector $\theta_x$.
Let $\zeta$ be a Gel'fand distribution  basis. Using the operator $W$ in (b) of Proposition \ref{prop_rieszbasis}, for $f\in \D$ one has
$W^{-1}f= \int_X \ip{W^{-1}f}{{\check{\zeta_x}}}\zeta_x d\mu.$ Hence $f=WW^{-1}f= \int_X \ip{W^{-1}f}{{\check{\zeta_x}}}\omega_x d\mu.$ But we also have
$f= \int_X \ip{f}{\theta_x}\omega_x d\mu.$ Hence, $\ip{W^{-1}f}{{\check{\zeta_x}}}=\ip{f}{\theta_x}$ $\mu$-a.e. Therefore,
$$Hf= \int_X \alpha(x) \ip{f}{\theta_x}\omega_x d\mu= W\int_X \alpha(x) \ip{f}{\theta_x}\zeta_x d\mu= W\int_X \alpha(x) \ip{W^{-1}f}{{\check{\zeta_x}}}\zeta_x d\mu.$$
Hence $W^{-1}f \in D(\overline{A})$, where $A$ is the operator defined by $\zeta$ as in Example \ref{ex_gelf}, and $Hf=W\overline{A}W^{-1}f$. From this one easily gets that $\overline{H}=W\overline{A}W^{-1}$.
Thus $\overline{H}$ and $\overline{A}$ are similar and therefore their spectra are the same.
This example, even though quite simple, has a certain relevance in the so-called {\em non-hermitian quantum mechanics} (see\cite{book1, book2} for overviews) where a non-hermitian hamiltonian is in fact defined as an operator similar to a self-adjoint hamiltonian. The similarity operator (in our example $W$) modifies the geometry of the Hilbert space and for this reason is named the {\em metric} operator.
\enex

\beex \label{ex_atomic}Another interesting situation where the interaction between frames and operators appears to play an important role arises when the action of an operator can be described through a Bessel sequence. This leads to the notion of an {\em atomic systems} for a given operator $A$, introduced first by Feichtinger and Werther in \cite{feich} and developed by \mbox{${\rm G\check{a}vru\c{t}a}$} in \cite{gavruta}.
Generalizing the notion given in those papers, we call a measurable distribution map $\omega$ an {\em atomic distribution map} for a given operator $A\in \LD$ if,
\begin{itemize}
\item[(a)] $\omega$ is a Bessel distribution map;
\item[(b)] for every $f\in \D$ there exists a function $\xi_f \in L^2(X, \mu)$ such that
$$ \ip{Af}{g} = \int_X \xi_f(x) \ip{\omega_x}{g}d\mu,\; \forall g\in \D;$$
\item[(c)] for every $x\in X$, the linear functional $f\to \xi_f(x)$ is continuous on $\D[t]$.
\end{itemize}
Once for every $f\in \D$ a function $\xi_f$ is selected so that (b) holds, using (c) one deduces that there exists $\theta_x\in \D^\times$ such $\xi_f(x)=\ip{f}{\theta_x}$. Hence,
$$ \ip{Af}{g} = \int_X \ip{f}{\theta_x} \ip{\omega_x}{g}d\mu,\; \forall f,g\in \D.$$

It is worth remarking that an atomic distribution map for $A\in \LD$ can be constructed starting  from a Gel'fand distribution basis $\zeta$. Indeed, putting $\omega_x:=(A^\dag)^\times \zeta_x$, $x\in X$ and proceeding as in \eqref{eqn_expA}, we get, for $f\in \D$,
$$\ip{Af}{g}=\int_X \ip{f}{\zeta_x}\ip{(A^\dag)^\times\zeta_x}{ g}d\mu=\int_X \ip{f}{\zeta_x}\ip{\omega_x}{ g}d\mu , \; \forall g \in \D$$
and
$$ \int_X |\ip{f}{\omega_x}|^2d\mu =\|A^\dag f\|^2 <\infty$$
\enex
Example \ref{ex_atomic} provides also some motivations for going beyond frames when looking for distribution functions that can {\em locally} describe  vectors of some domain $\D$ in Hilbert spaces or the action of operators defined on it. For sequences and Hilbert-space valued measurable functions there are several generalizations such as {\em pseudoframes} \cite{li_ogawa} or {\em reproducing pairs} \cite{bal_ms, jpa_ms_ct} . We will discuss this topic in the next section.

\enex

\section{Coefficient spaces and duality}\label{sect_5}
A Bessel distribution map $\omega$, as defined in Section \ref{sect_3}, is characterized by the fact that the analysis operator $D_\omega: f\in \D \to \ip{f}{\omega_x}$ maps every $f\in \D$ into $L^2(X, \mu)$.
As we shall see in what follows the analysis operator and the synthesis operators can be introduced for an arbitrary $\D^\times$-valued function $\omega$, and they determine some spaces of measurable functions that, under certain circumstances, enjoy nice duality properties. In the case of Hilbert spaces the problem of {\em coefficient spaces} was studied \cite{jpa_ms_ct} and then generalized in \cite{jpa_ct_rppip} to rigged Hilbert spaces and more general {\sc Pip}-spaces. A short summary of this construction is given below for the reader's convenience.

\subsection{Construction of coefficient spaces}
Let  $\omega:x\in X\to \D^\times$ be a weakly measurable function.
We denote by $V_\omega$ the space of all measurable functions $\xi: X\to {\mb C}$ such that the integral
$$ F_\omega(g):=\int_X \xi(x) \ip{\omega_x}{g} d\mu $$ exists for every $g\in \D$ and defines a continuous conjugate linear functional on $\D[t].$
As usual we do not distinguish functions $\xi$ which differ on $\mu$-null subsets of $X$.
We refer sometimes to $V_\omega$ as the coefficient space of $\omega$.

Then we can define a linear map
$T_\omega:V_\omega \to \D^\times$, which we call again the {\em synthesis operator} by the following relation
\begin{equation}\label{eqn_Tomega} \ip{T_\omega\xi}{g}= \int_X \xi(x) \ip{\omega_x}{g} d\mu, \quad \xi \in V_\omega, g \in \D.\end{equation}

Set
 $W_\omega = V_\omega /{\rm Ker}\,T_\omega$ and $[\xi]_\omega:=\xi + {\rm Ker}\,T_\omega$. Obviously, $\omega$ is $\mu$-independent if, and only if, ${\rm Ker}\,T_\omega=\{0\}$; that is, if and only if, $W_\omega = V_\omega$.

We introduce on $W_\omega$ a topology $\tau_\omega$ by means of the following set of seminorms:
\begin{equation}\label{seminW} p_\M ([\xi]_\omega) = \sup_{g\in \M} |\ip{g}{T_\omega\xi}|  = \sup_{g\in \M} \left| \int_X \overline{\xi(x)} \ip{g}{\omega_x} d\mu\right|,\end{equation}
where $\M$ runs in the family of bounded subsets of $\D[t]$. 

By \eqref{semin_Dtimes} and \eqref{seminW} the following equality holds {for all bounded subsets $\M\subset\D$}:
\begin{equation}\label{eqn_twotopol}
p_\M ([\xi]_\omega)= q_\M ({ T}_\omega\xi), \quad \forall \xi \in V_\omega,
\end{equation}
i.e., {${ \hat{T}}_\omega[\xi]_\omega:={ T}_\omega\xi$} is continuous, injective and, if $W_\omega$ is complete, it has closed image and its inverse is also continuous (but not necessarily everywhere defined).

As shown in \cite[Theorem 3.4]{jpa_ct_rppip},
if $\D[t]$ is a reflexive space and $\omega:X\to \D^\times$ a weakly continuous map, a linear functional $H$ on $W_\omega[\tau_\omega]$ is continuous if, and only if, there exists $g\in \D$ such that
$$ H([\xi]_\omega) = \int_X \xi(x) \ip{\omega_x}{g} d\mu, \quad \forall \xi \in V_\omega;$$
i.e., $H=H_g$.
In other words, the dual space $W_\omega^*$ of $W_\omega$, with respect to the sesquilinear form given by the $L^2$ inner product, can be identified with a space $E$ of measurable functions containing all functions $\{\ip{g}{\omega_x}, g\in \D\}$. Our next scope is to determine $E$ explicitly, at least under certain conditions. In what follows, if $E$ is a linear space of measurable functions, we denote by $\overline{E}$ the linear space consisting of the complex conjugate functions of $E$.

{
\berem
Since $\hat{T}_\omega $ is continuous from $W_\omega$ into $\D^\times$ and $\D$ is reflexive, it admits an adjoint $\hat{T}_\omega^\dag: \D \to W_\omega^*$, the dual space of $W_\omega$.

If we define a linear operator $D_\omega$ by $(D_\omega g)(x)=\ip{g}{\omega_x}$, the previous equation reads as
$$\ip{[\xi]_\omega}{\hat{T}_\omega^\dag g}=\int_X \xi(x) \overline{(D_\omega g)(x)} d\mu, \quad \forall \xi \in V_\omega, \,g\in \D.$$
We call $D_\omega$, as customary, the {\em analysis operator} associated to $\omega$.

\enrem

\beex
Let us consider the space of tempered distributions $\mathcal S^\times(\mathbb R)$, and  $\omega_x\in \mathcal S^\times(\mathbb R)$ as the Dirac delta centered at $x$: i.e. $\omega_x=\delta_x$. Then, the coefficient space $V_\delta$ is defined as  the space of all measurable functions $\xi:\mathbb R\rightarrow\mathbb C$ such that $\xi(x)g(x)\in L^1(\mathbb R)$ for all $g(x)\in \mathcal S(\mathbb R)$ and such that $F_\delta(g):=\int\xi(x)g(x)d\mu$ is continuous, i.e. $F_\delta\in\mathcal S^\times(\mathbb R)$. Let us consider the set $C_M(\mathbb R)$   of polynomially bounded continuous  functions  on $\mathbb R$, i.e. $f\in C_M(\mathbb R)$ if and only if there exists $C>0$ and $N\in\mathbb N$ such that
$$
|f(x)|\leq C [1+x^2]^N, \quad \forall x\in\mathbb R.
$$
We prove that $\xi(x)\in V_\delta$ if, and only if, there exists  a  $f\in C_M(\mathbb R)$ and $k\in\mathbb N$ such that $\xi(x)=\partial^kf(x):=f^{(k)}(x)$ (the derivative is intended in \textit{weak} sense). Indeed, let us consider a $k$-times differentiable function $f\in C_M(\mathbb R)$. By definition one has
$$
\int_{\mathbb R} f^{(k)}(x)g(x)d\mu=(-1)^k\int_{\mathbb R} f(x) g^{(k)}(x)d\mu, \quad\forall g\in\mathcal S(\mathbb R).
$$
The integral on the right hand side is convergent; then, $f^{(k)}(x)g(x)\in L^1(\mathbb R)$ for all $g(x)\in\mathcal S(\mathbb R)$. In this way, a functional $F_\delta$ on $\mathcal S(\mathbb R)$ is defined. Its continuity follows from an easy  estimate:
$$
{\Bigr |}\int_{\mathbb R} f(x) g^{(k)}(x)d\mu{\Bigl |}\leq\int_{\mathbb R}|(1+x^2)^{-1} [(1+x^2)f(x) g^{(k)}(x)]|d\mu\leq\pi C\sup_{x\in\mathbb R}|(1+x^2)^{N+1} g^{(k)}(x)|.
$$
Then $f^{(k)}(x)\in V_\delta$. Conversely, let us consider $F_\delta\in\mathcal S^\times(\mathbb R)$. Then, by the regularity theorem of distributions \cite[Vol. I]{reed1}, there exists $f\in C_M(\mathbb R)$ and $k\in\mathbb N$ such that
$$
F_\delta[g]:=\int_{\mathbb R}\xi(x)g(x)d\mu=(-1)^k\int_{\mathbb R}f(x) g^{(k)}(x)d\mu,\quad\forall g\in\mathcal S(\mathbb R)
.$$
This means that $f(x)$ is $k$-times differentiable and that $\xi(x)=f^{(k)}(x)$.\\

\enex

\subsection{Compatible pairs}\label{subsect_compatible}
If $\omega, \theta$ are two weakly measurable $\D^\times$-valued functions, then one can formally define a sesquilinear form on $\D\times \D$ by
\begin{equation}\label{eqn_lambda} \Omega_{\theta, \omega} (f, g) = \int_X \ip{f}{\theta_x} \ip{\omega_x}{g}d\mu, \quad f,g \in \D.\end{equation}
 If $\Omega_{\theta, \omega}$ is well defined on $\D\times\D$ and jointly continuous,
  there exists an operator $S_{\theta, \omega}$, which maps $\D[t]$ into $\D^\times[t^\times]$ continuously,  such that:
\begin{equation}\label{eqn_defnS} \ip{S_{\theta, \omega} f}{g} = \int_X \ip{f}{\theta_x} \ip{\omega_x}{g}d\mu, \quad f,g \in \D.\end{equation}

In  this case, as proved in \cite[Theorem 3.6]{jpa_ct_rppip}, the dual $W_\omega^*$ can be identified with a closed subspace of $\overline{W}_\theta$, the space of conjugates of elements of $W_\omega$. On the other hand, in \cite{jpa_ms_ct} it was proved that, for a {\em reproducing pair} of $\H$-valued weakly measurable functions $\omega, \theta$ (this means that $\Omega_{\theta, \omega}$ is bounded and the corresponding operator $S_{\theta, \omega}$ is bounded and has bounded inverse), the  spaces $W_\omega$ and $W_\theta$ are both Hilbert spaces in conjugate duality to each other. The analogous statement for $\D^\times$-valued measurable functions was left open in \cite{jpa_ct_rppip}. In this section we want to discuss further this question.

{\bedefi Let $\omega, \theta$ be   weakly measurable $\D^\times$-valued functions such that the sesquilinear form $\Omega_{\theta, \omega}$ in \eqref{eqn_lambda} is well defined in $\D\times \D$ and jointly continuous. We say that $\omega, \theta$ are {\em compatible} if $W_\omega^*$ is topologically isomorphic to $\overline{W}_\theta$ and $W_\theta^\times$ is topologically isomorphic to $W_\omega$.
Thus we identify $W_\omega^*$ with $\overline{W}_\theta$ and $W_\theta^\times$ with $W_\omega$.
\findefi

This definition implies that the spaces $W_\omega$ and $W_\theta$ are reflexive spaces enjoying the duality properties mentioned above which we write shortly as $W_\omega^*\approx \overline{W}_\theta$ and $W_\theta^\times\approx W_\omega$. Then we have
$$ W_\theta^* \approx \overline{W_\theta^\times}\approx \overline{W_\omega^*}\;\mbox{ and }\; W_\omega^\times \approx \overline{W_\theta^*}\approx W_\theta.$$
This proves that  $(\omega,\theta)$ is a compatible pair if and only if $(\theta,\omega)$ is a compatible pair.  Theorem \ref{thm_compatible} below makes this symmetry clearer, taking also into account the fact that
{the map $S_{\theta, \omega}$ has an adjoint $S_{\theta, \omega}^\dag$ defined by
$$ \ip{S_{\theta, \omega}^\dag f}{g}= \overline{\ip{S_{\theta, \omega}g}{f}}, \quad f,g \in \D.$$
An easy computation shows that
$$ \ip{S_{\theta, \omega}^\dag f}{g}=\int_X \ip{f}{\omega_x} \ip{\theta_x}{g}d\mu, \quad f,g \in \D.$$
Hence, $S_{\theta, \omega}^\dag=S_{\omega, \theta}$.}

\begin{thm} \label{thm_compatible} Let $\D[t]$ be a reflexive Fr\`{e}chet space and $\omega, \theta$ two weakly measurable functions with $\omega, \theta$  $\mu$-independent and total. {If $V_\omega$ is complete}, then the following statements are equivalent:
\begin{itemize}
\item[(i)] ($\omega,\theta$) is a {compatible pair};

\item[(ii)]The operator $S_{\theta, \omega}$ is a topological isomorphism of $\D[t]$ onto $\D^\times[t^\times]$.
\end{itemize}
\end{thm}
\begin{proof}
{We begin with observing that, since $\omega, \theta$ are $\mu$-independent, $W_\omega=V_\omega$,  $W_\theta=V_\theta$, ${\hat{T}_\omega=T_\omega}$.}

(i) $\Rightarrow$ (ii): If ($\omega,\theta$) is a compatible pair, then, as remarked in \eqref{eqn_defnS} the operator $S_{\theta, \omega}$ is continuous from $\D[t]$ into $\D^\times[t^\times]$. The $\mu$-independence of $\omega$  and the fact that $\theta$ is total imply that $S_{\theta, \omega}$ is injective. 
In order to show that $S_{\theta, \omega}$ is also surjective, we begin with proving that ${\sf Ran}{T}_\omega$ is dense in $\D^\times[t^\times]$. Suppose that $g\in \D=\D^{\times\times}$ is such that $\ip{{T}_\omega\xi}{g}=0$ for every $\xi \in V_\omega$. Thus,
$$ \ip{{T_\omega}\xi}{g}=\int_X \xi(x)\ip{\omega_x}{g}d\mu =0, \quad \forall \xi \in V_\omega.$$
In particular this is true if we take $\xi(x)=\ip{f}{\theta_x}$, with $f\in \D$. The $\mu$-independence of $\theta$ implies that $\ip{\omega_x}{g}=0$ for almost every $x\in \RN$. Hence $ g=0$, because $\omega$ is total.

From \eqref{eqn_twotopol} and using the fact that the dual of a Fr\`{e}chet space is complete, it follows easily that ${\sf Ran}{T}_\omega$ is closed in $\D^\times$.
Hence, ${\sf Ran}{T}_\omega=\D^\times$. Therefore, for every $\Phi\in \D^\times$ there exists a unique $\xi\in V_\omega$ such that ${T}_\omega \xi=\Phi$, this implies that
$$ \ip{\Phi}{g}= \int_X \xi(x)\ip{\omega_x}{g}d\mu, \quad \forall g\in \D.$$
Let us now consider the conjugate linear functional $F_\xi$ on $V_\theta$ defined by:
$$F_\xi(\eta)= \int_X \xi(x)\overline{\eta(x)}d\mu, \quad \eta\in V_\theta.$$ Then, there exists $f \in \D$ such that
$$ F_\xi(\eta)= \int_X \ip{f}{\theta_x}\overline{\eta(x)}d\mu, \quad \eta\in V_\theta.$$
Hence, in particular
$$ F_\xi(\ip{g}{\omega_x})= \int_X \ip{f}{\theta_x}\ip{\omega_x}{g}d\mu, \quad \forall g\in \D.$$
In conclusion,
$$ \ip{\Phi}{g}= \int_X \ip{f}{\theta_x}\ip{\omega_x}{g}d\mu, \quad \forall g\in \D.$$
This implies that $\Phi= S_{\theta, \omega}f$.
{It remains to prove that $S_{\theta, \omega}^{-1}$ is continuous. For this we observe that, by a symmetry argument, $S_{\omega, \theta}$ is also surjective. But, as we have seen above, $S_{\theta, \omega}^\dag=S_{\omega, \theta}$. Hence, $S_{\theta, \omega}^\dag\D=\D^\times$. Then, by \cite[\S 38, n.4 (5)]{Kothe}, $S_{\theta, \omega}^{-1}$ is continuous.}

(ii) $\Rightarrow$ (i): Let $H\in V_\omega^\times$; then, there exists $g \in \D$ such that:
$$H(\xi) = H_g(\xi) = \int_X \xi(x) \ip{\omega_x}{g} d\mu, \quad \forall \xi \in V_\omega.$$
We show that $g$ is unique. Suppose that there exists another $g'$ satisfying the same condition.
Then it follows that:
$$ \int_X \xi(x) \ip{\omega_x}{g-g'} d\mu =0, \quad \forall \xi \in V_\omega.$$
In particular, this is true for $\xi(x)= \ip{f}{\theta_x}$, for every $f \in \D$.
Thus:
$$\int_X\ip{f}{\theta_x}\ip{\omega_x}{g-g'} d\mu =0, \quad \forall \xi \in V_\omega.$$
Hence by \eqref{eqn_defnS} we get $\ip{S_{\theta, \omega} f}{g-g'}=0$, for every $f\in \D$. But, since $S_{\theta, \omega}\D=\D^\times$, we conclude that $\ip{G}{g-g'}=0$, for every $G\in \D^\times$ and this, in turn, implies that $g=g'$.
Therefore, we can define a map $J: H\in V_\omega^\times \to g \in \D$, where $g$ is the unique element of $\D$ such that $H=H_g$. This map is an isomorphism of vector spaces. Indeed, it is clearly injective.
On the other hand, if $g \in \D$, the functional:
$$H(\xi) = H_g(\xi) = \int_X \xi(x) \ip{\omega_x}{g} d\mu, \quad  \xi \in V_\omega,$$
is in $V_\omega^\times$ and satisfies $J(H_g)=g$.
It is clear that the function $\eta(x)=\ip{g}{\omega_x}$ is an element of $V_\theta$. Assume that there exists another function $\eta' \in V_\theta$ such that
$$ H(\xi) = \int_X \xi(x) \overline{\eta'(x)} d\mu, \quad  \xi \in V_\omega.$$
This implies that:
$$\int_X \xi(x) \overline{(\eta(x)-\eta'(x))} d\mu=0, \quad \forall \xi \in V_\omega.$$
Take $\xi(x)=\ip{f}{\theta_x}$, $f \in \D$, then $\eta-\eta' \in {\rm Ker}T_\theta$. Hence $\eta = \eta'$.
Then we can define a linear map:
\begin{equation}\label{mapJ} Y: H\in V_\omega^* \to \eta \in V_\theta,\end{equation}
where $V_\omega^*$ denote the dual space of $V_\omega$. The map $Y$ is clearly injective.
As we have seen $T_\omega^\times f=\ip{f}{\omega_x}\in V_\theta$, for every $f\in \D$. Now we put $C_{\omega,\theta}f =T_\omega^\times f$. Hence $C_{\omega,\theta}$ is a linear map of $\D$ into $V_\theta$.
We want to prove that ${\sf Ran\,}C_{\omega,\theta}$ coincides with $V_\theta$.

Let $p$ be a seminorm defining the topology of $\D$. Then, there exists a bounded subset $\M$ of $\D$ and $c>0$ such that
\begin{align*}
p(f)&=p (S_{\omega, \theta}^{-1}S_{\omega,\theta}f)\leq c\, q_\M(S_{\omega,\theta}f)=c\,\sup_{g\in \M}|\ip{S_{\omega,\theta}f}{g}|\\
&=c\,\sup_{g\in \M}\left|\int_X \ip{f}{\omega_x} \ip{\theta_x}{g}d\mu \right|=c\,p_\M(\ip{f}{\omega_x}),
\end{align*}
where we have used the seminorms defined in \eqref{semin_Dtimes} and in \eqref{seminW}.
This inequality easily implies that ${\sf Ran\,}C_{\omega,\theta}$ is closed in $V_\theta$.

Now, we prove that ${\sf Ran\,}C_{\omega,\theta}$ is also dense in $V_\theta$. Were it not so there would be a nonzero continuous linear functional $F$ on $V_\theta$ such that $F(\ip{f}{\omega_x})=0$ for every $f \in \D$.
Hence, there exists $g\neq 0$ such that:
$$ F(\xi)=\int_X\xi(x)\ip{\theta_x}{g}d\mu, \forall \xi \in V_\theta,$$
and therefore
$$F(\ip{f}{\omega_x})=\int_X\ip{f}{\omega_x}\ip{\theta_x}{g}d\mu, \forall f\in \D.$$
This implies that $\ip{S_{\omega,\theta}f}{g}=0$ for all $f\in \D$. But, since $S_{\omega,\theta}\D=\D^\times$, this produces a contradiction.
Let us now consider the map $J$ defined in \eqref{mapJ}. An immediate consequence of the equality ${\sf Ran\,}C_{\omega,\theta}=V_\theta$ is that $J$ is surjective.
This implies that the conjugate dual of $V_\omega$ can be identified with  $\overline{V}_\theta$, the space of conjugates of elements of $V_\theta$.
\end{proof}
\berem Using an argument similar to that of \cite[Proposition 3.10]{gb_ct_riesz} one can prove that, if a topological isomorphism between $\D$ and $\D^\times$ exists, then $\D^\times$ can be made into a Hilbert space. Thus the rigged Hilbert space reduces to a triplet of Hilbert spaces.
\enrem

\section*{Concluding remarks} The discussion in Section \ref{sect_4} gives some hints for introducing the notion of {\em reproducing pairs} of distribution functions. In Hilbert spaces, this means that the operator corresponding to $S_{\theta,\omega}$, which is defined similarly to what we have done here, is an element of $GL(\H)$, the group of bounded operators having bounded inverse in $\H$. Hence, it provides a topological isomorphism of $\H$ onto itself. As we have seen, at least in the case we have considered (the functions $\omega$ and $\theta$ are both $\mu$-independent and total), there is an equivalence between $S_{\theta,\omega}$ being a topological isomorphism of $\D$ onto $\D^\times$ and $\theta, \omega$ forming a compatible pair. This imposes nice conditions of duality between the spaces $W_\omega$ and $W_\theta$. This was exactly what was obtained in \cite{jpa_ms_ct} in the case of reproducing pairs  taking their values in Hilbert spaces and also discussed (but not completely achieved)  in \cite{jpa_ct_rppip} in the case that the functions take values in rigged Hilbert spaces (or, more generally, in partial inner product spaces). A complete identification of the notions of compatible pairs and reproducing pairs is however not yet obtained. We hope to come back to this question in a future paper.

\section*{Acknowledgement} {The authors thank the referees for their useful comments and suggestions. This paper has been made within the framework of the project INdAM-GNAMPA 2018 {\em Alcuni aspetti di teoria spettrale di operatori e di algebre; frames in spazi di
Hilbert rigged. }}

\bibliographystyle{amsplain}

\end{document}